\documentclass{article}
\usepackage{amssymb,amsmath,graphicx,amsthm,mathrsfs}
\usepackage{graphicx,subfigure,float,url,color}
\usepackage{mathrsfs}
\usepackage[colorlinks=true]{hyperref}
\usepackage{pdfsync}

\hypersetup{linkcolor=blue,urlcolor=blue,citecolor=red}
\newcommand{\red}[1]{{\bf\color{red}{#1}}}

\theoremstyle{plain}
\newtheorem{theorem}{Theorem}

\newtheorem{proposition}{Proposition}
\theoremstyle{definition}
\newtheorem{definition}{Definition}
\newtheorem{example}{Example}

\newtheorem{remark}{Remark}
\numberwithin{equation}{section}
\title{
\textbf{Integral and measure-turnpike properties for infinite-dimensional optimal control systems}
}

\renewcommand{\thefootnote}{\arabic{footnote}}

\author{\renewcommand{\thefootnote}{\arabic{footnote}} Emmanuel Tr\'elat\,\footnotemark,\;\;
Can Zhang\,\footnotemark$\phantom{\;}^,$\footnotemark}

\begin{document}
%%%%%%%%%%%%%%%%%%%%%%%%%%%%%%%%%%%%%%%%%%%%%%%%%%
\maketitle

\renewcommand{\thefootnote}{\arabic{footnote}}
\footnotetext[1]{Sorbonne Universit\'es, UPMC Univ Paris 06, CNRS UMR 7598, Laboratoire Jacques-Louis Lions, F-75005, Paris, France}
\footnotetext[2]{School of Mathematics and Statistics, Wuhan University, 430072 Wuhan, China}
\footnotetext[3]{Sorbonne Universit\'es, UPMC Univ Paris 06, CNRS UMR 7598, Laboratoire Jacques-Louis Lions, F-75005, Paris, France}

%%%%%%Email address %%%%%%%%%%%%%%%%%%%%%%%%%%%%%%%%%%%%%%%%%%%%%%
\let\thefootnote\relax\footnotetext{Email address: emmanuel.trelat@upmc.fr (Emmanuel Tr\'elat), zhangcansx@163.com (Can Zhang)
}
%%%%%%%%%%%%%%%%%%%%%%%%%%%%%%%%%%%%%%%%%%%%%%%%%%%%

%\date{}
\maketitle

\begin{abstract}

We first derive a general integral-turnpike property around a set 
for infinite-dimensional non-autonomous optimal control problems with any possible terminal state constraints, under some appropriate assumptions. Roughly speaking, the integral-turnpike property means that the time average of the distance from any optimal trajectory to the turnpike set converges to zero, as the time horizon tends to infinity.
Then, we establish the measure-turnpike property for strictly dissipative optimal control systems, with state and control constraints. The measure-turnpike property,
which is slightly stronger than the integral-turnpike property,
 means that any optimal (state and control) solution remains essentially, along the time frame, close to an optimal solution of an associated static optimal control problem, except along a subset of times that is of small relative Lebesgue measure as the time horizon is large. 
Next, we prove that strict strong duality, which is a classical notion in optimization, implies strict dissipativity, and measure-turnpike.
Finally,  we conclude the paper with several comments and open problems.

 \end{abstract}

\bigskip


\bigskip


%%%%%%%%%%%%%%%%%%%%%%%%%%%%%%%%%%%%%%%%%%%%%%%%%%%%
\section{Introduction}%\label{Intro}
We start this paper  with an intuitive idea in general terms. 
Consider the optimal control problem
\begin{equation*}\boxed{
\begin{split}
& \inf \frac{1}{T}\int_0^Tf^0(y(t),u(t))\,dt,\\
& \text{subject to}\;\; \dot{y}(t)=f(y(t),u(t)), \quad t\in[0,T], \\[2mm]
\end{split}}
\end{equation*}
under some terminal state conditions, with $T>0$ large. Setting $s=t/T$ and $\varepsilon=1/T$, we rewrite the above optimal control problem as
\begin{equation*}\boxed{
\begin{split}
& \inf \int_0^1f^0(y(s),u(s))\,ds,\\
&  \text{subject to}\;\;\varepsilon \dot y(s) = f(y(s),u(s)),\quad s\in[0,1].\\[2mm]
\end{split}}
\end{equation*}
Then, we expect that, as $\varepsilon\rightarrow 0$, there is some convergence to the static problem
$$
\boxed{
\inf f^0(y,u), \;\;\text{subject to}\;\;f(y,u)=0. \\[2mm]
}$$
This intuition has been turned into rigorous results in the literature, under some appropriate assumptions.  These results say roughly that, if $T$ is large, then any optimal solution $y(\cdot)$ on $[0,T]$ spends most of its time close to an optimal solution $ y_s$ of the static problem. 
This is the (neighborhood) turnpike phenomenon. We call the point $y_s$ a turnpike point.

This turnpike phenomenon was first observed and investigated by economists for discrete-time optimal control problems 
(see, e.g., \cite{DorfmanSamuelsonSolow,Mc}). In the last three decades, many turnpike results have been established in a large number of works (see, e.g., \cite{Kokotovic, AL, CHJ, CarlsonBOOK,  Grune1, Faulwasser1, GTZ, LWei, Rapaport, Z1, Z2, Z3}  and references therein), either for discrete-time or continuous-time problems 
involving control systems in finite-dimensional state spaces, and very few of them in the infinite dimensional setting.

A more quantitative turnpike property, which is called the exponential turnpike property, has been established in \cite{PZ1,PZ2,TZ1} for both the linear and nonlinear continuous-time optimal controlled systems.
It means  that the optimal solution for the dynamic controlled problem remains exponentially close to an optimal solution for the corresponding static controlled problem  within a sufficiently large time interval contained in the long-time horizon under consideration. We stress that in those works not only the optimal state and control, but also the corresponding adjoint vector, resulting from the application of the Pontryagin maximum principle, were shown to remain exponentially close to an extremal triple for a corresponding static optimal control problem, except at the extremities of the time horizon. The main ingredient in the papers \cite{PZ1,PZ2,TZ1}
is an exponential dichotomy transformation and the hyperbolicity feature of the Hamiltonian system, deriving from the Pontryagin maximum principle, under some controllability and observability assumptions.

However, not all turnpike phenomena are around a single point. For instance,
the turnpike theorem for calculus of variations problems in \cite{Rapaport} is proved for the case when there are several turnpikes.  More precisely, they
show that 
there exists a competition between the several turnpikes for optimal trajectories with different initial states,
and provide in particular a criterion for the choice of turnpikes that are in competition.
On the another hand, for  some classes of optimal control problems for periodic systems, the turnpike phenomenon may occur around a periodic trajectory, which is itself characterized as being the optimal solution of an appropriate periodic optimal control problem (cf., e.g., \cite{Sa, TZZ, Z2, Z3, Zon}).
\medskip

In this paper, the first main result is to derive a more general turnpike result, valid for very general classes of optimal control problems settled in an infinite-dimensional state space, and where the turnpike phenomenon  is around a set $\mathcal{T}$. This generalizes the standard case where $\mathcal{T}$ is a singleton, and the less standard case where $\mathcal{T}$ is a periodic trajectory.  Between the case of one singleton and the periodic trajectory,  however, 
there are, to our knowledge, very few examples for intermediate situations  in the literature.

The organization of the paper is as follows. 
In Section \ref{general}, we build up an abstract framework to derive a general turnpike phenomenon around a set.  
In Section \ref{sec_turnpike}, we enlighten  the relationship between the above-mentioned abstract framework 
and the strict dissipativity property. Under the strict dissipativity assumption for optimal control problems, we establish the so-called measure-turnpike property. 
In Section \ref{sec_dissip}, we provide some material to clarify the relationship between measure-turnpike, strict dissipativity and strong duality.  Finally, Section \ref{consec} concludes the paper.
\section{An abstract setting}\label{general}

In this section, we are going to derive a general turnpike phenomenon around a set $\mathcal{T}$. The framework is the following.

Let $X$ (resp., $U$) be a reflexive Banach space endowed with the norm $\|\cdot\|_X$ (resp., $\|\cdot\|_U$).  
 Let $f: \mathbb R \times X\times U\rightarrow X $ be a continuous mapping that is uniformly  Lipschitz continuous in $(y,u)$ for all $t\in\mathbb R$. Let $f^0: \mathbb R\times X\times U\rightarrow \mathbb R$ be a continuous function that is bounded from below.
Let $E$ and $F$ be two subsets of $X$ and $U$, respectively. 
Given any $t_0\in\mathbb R$ and $t_1\in\mathbb R$ with $t_0<t_1$, we consider the non-autonomous  optimal control problem
$$\boxed{
(P_{[t_0,t_1]})\qquad \left\{\begin{array}{l}
J_{[t_0,t_1]} = \inf\frac{1}{t_1-t_0}  \int_{t_0}^{t_1} f^0(t,y(t),u(t))\, dt,\\[2mm]
\text{subject to}\;\;\;\;\dot y(t) = A(t)y+f(t,y(t),u(t)),\quad t\in[t_0,t_1],\\[2mm]
R(t_0,y(t_0),t_1,y(t_1))=0,\quad (y(t),u(t))\in E\times F, \quad t\in[t_0,t_1].\\[2mm]
\end{array}\right.}
$$
Here, 
$(A(t), D(A(t)))$ is a family of unbounded operators on $X$ such that the existence of the corresponding two-parameter evolution system $\Phi (t, s)$ is ensured (cf., e.g.,
\cite[Chapter 5, Definition 5.3]{Pa}), 
the controls are Lebesgue measurable functions $u(\cdot) : [t_0,t_1] \rightarrow F$, and $Y$ is a Banach space, the mapping $R: \mathbb R\times X\times\mathbb R\times X\rightarrow Y$ stands for any possible terminal state conditions.
Throughout the paper, the solutions $(y(\cdot),u(\cdot))\in C([t_0,t_1];X)\times L^2(t_0,t_1;U)$ are considered in the mild sense, meaning that
$$y(\tau)=\Phi(\tau,t_0)y(t_0)+\int_{t_0}^\tau \Phi(\tau,t) f(t,y(t),u(t))\,dt,\qquad\forall\tau\in[t_0,t_1].$$

\begin{remark}
\emph{
Typical examples of terminal conditions are the following:
\begin{itemize}
\item  When both initial and final conditions are let free in $(P_{[t_0,t_1]})$, take $R=0$.
\item  When the initial point is fixed (i.e., $y(t_0)=y_0$) and the final point is let free, take $R(s_0,z_0,s_1,z_1)=z_0-y_0$.
\item When both initial and final conditions are fixed (i.e., $y(t_0)=y_0$ and $y(t_1)=y_1$), take $R(s_0,z_0,s_1,z_1)=(z_0-y_0,z_1-y_1)$.
\item When the final point is expected to coincide with the initial point (i.e., $y(t_0)=y(t_1)$ without any other constraint), for instance in a  periodic optimal control problem, in which one assumes that there exists 
$T>0$ such that $f(t+T,y,u)=f(t,y,u)$ and $f^0(t+T,y,u)=f^0(t,y,u)$, $\forall (t,y,u)\in \mathbb R\times X\times U$,
take $R(s_0,z_0,s_1,z_1)=(s_1-s_0-T, z_0-z_1)$.
\end{itemize}
}
\end{remark}

Hereafter, we call $(y(t),u(t))$, $t\in[t_0,t_1]$, an admissible pair if it verifies the state equation and the constraint $(y(t),u(t))\in E\times F$ for almost every $t\in[t_0,t_1]$.  We remark that the definition of admissible pair does not require that the terminal state condition $R(t_0,y(t_0),t_1,y(t_1)) = 0$ is satisfied.
We denote by  $$C_{[t_0,t_1]}(y(\cdot),u(\cdot)) = \int_{t_0}^{t_1} f^0(t,y(t),u(t))\, dt$$ the cost of an admissible pair $(y(\cdot),u(\cdot))$ on $[t_0,t_1]$.
In other words, $J_{[t_0,t_1]}$ is the infimum with time average cost ({\it Ces\`aro mean}) over all admissible pairs satisfying the constraint on terminal points:
$$
J_{[t_0,t_1]} = \inf \left\{ \frac{1}{t_1-t_0}C_{[t_0,t_1]}(y(\cdot),u(\cdot)) \ \mid\ (y(\cdot),u(\cdot))\textrm{ admissible},\ R(t_0,y(t_0),t_1,y(t_1))=0 \right\}.
$$
Throughout the paper, we assume that the problem $(P_{[t_0,t_1]})$ has optimal solutions, and that an admissible pair $(y(\cdot),u(\cdot))$, with initial state $y(t_0)$,  is said to be optimal  for the problem $(P_{[t_0,t_1]})$ if $R(t_0,y(t_0),t_1,y(t_1))=0$ and $\frac{1}{t_1-t_0}C_{[t_0,t_1]}(y(\cdot),u(\cdot))=J_{[t_0,t_1]}$.
Existence of optimal solutions for optimal control problems  is well-known under appropriate convexity assumptions on $f^0$, $f$ and $R$ with $E$ and $F$ convex and closed (see, for instance, \cite[Chapter 3]{LiXunjing}). 

\medskip

We then consider the optimal control problem
$$\boxed{
(\bar P_{[t_0,t_1]})\qquad \left\{\begin{array}{l}
\bar J_{[t_0,t_1]} = \inf\frac{1}{t_1-t_0}C_{[t_0,t_1]}(y(\cdot),u(\cdot)),\\[2mm]
\text{subject to}\;\;\;\;\dot y(t) =A(t)y+ f(t,y(t),u(t)), \quad t\in[t_0,t_1],\\[2mm]
\quad (y(t),u(t))\in E\times F,\quad t\in[t_0,t_1].\\[2mm]
\end{array}\right.}
$$
Compared with the problem $(P_{[t_0,t_1]})$, in the above problem there is no terminal state constraint, i.e., $R(\cdot)=0$.
In fact, it is the infimum with time average cost  over all possible admissible pairs:
$$
\bar J_{[t_0,t_1]} = \inf \left\{ \frac{1}{t_1-t_0}C_{[t_0,t_1]}(y(\cdot),u(\cdot)) \ \mid\ (y(\cdot),u(\cdot))\textrm{ admissible} \right\}.
$$
We say the problem $(\bar P_{[t_0,t_1]})$ has a limit value if $\lim_{t_1\rightarrow+\infty}\bar J_{[t_0,t_1]} $
exists. 
We refer \cite{GR, QR} for the sufficient conditions ensuring the existence of the limit value. 
More precisely, asymptotic properties of optimal values, as $t_1$ tends to infinity, have been studied in \cite{QR} under suitable nonexpansivity  assumptions, and in \cite[Corollary 4 (iii)]{GR} by using occupational measures. 
In the sequel, we assume it exists and is written as 
$$
\bar J_{[t_0,+\infty)} = \lim_{t_1\rightarrow +\infty} \bar J_{[t_0,t_1]}.
$$

Besides, given any $y\in X$ we define the value function
$$
%\boxed{
V_{[t_0,t_1]}(y) = \inf \left\{ \frac{1}{t_1-t_0}C_{[t_0,t_1]}(y(\cdot),u(\cdot)) \ \mid\ (y(\cdot),u(\cdot))\textrm{ admissible},\ y(t_0)=y \right\}.
%}
$$
It is the optimal value of the optimal control problem with fixed initial data $y(t_0)=y$ (but free final point). Note that, if there exists no admissible trajectory starting at $y$ (because $E$ would not contain $y$), then we set $V_{[t_0,t_1]}(y)=+\infty$.
For each $y\in X$, we say a limit value exists if $\lim_{t_1\rightarrow+\infty} V_{[t_0,t_1]}(y) $
exists.   We now assume that, for each $y\in X$, the limit value exists  and is written as
$$
V_{[t_0,+\infty)}(y) = \lim_{t_1\rightarrow +\infty}  V_{[t_0,t_1]}(y).
$$

Clearly, we have
$$
\forall t_0<t_1,\qquad J_{[t_0,t_1]}\geq \bar J_{[t_0,t_1]},
$$
and thus
\begin{equation}\label{lower limit}
\liminf_{t_1\rightarrow +\infty} J_{[t_0,t_1]}\geq \bar J_{[t_0,+\infty)}.
\end{equation}
Meanwhile,
$$
\forall t_0<t_1,\quad\forall y\in X,\qquad   V_{[t_0,t_1]}(y)\geq \bar J_{[t_0,t_1]},
$$
and thus
$$
\forall y\in X,\qquad  V_{[t_0,+\infty)}(y)\geq  \bar J_{[t_0,+\infty)}.
$$

\medskip

 \begin{remark}\label{inv}
\emph{If the optimal control problem is autonomous (i.e., $A(\cdot)=A$, $f$ and $f^0$ are independent of time variable), it follows from the definitions that $\bar J_{[t_0,+\infty)}$, as well as
$V_{[t_0,+\infty)}(y)$, $\forall y\in X$, do not depend on $t_0\in  \mathbb R$.}
\end{remark}

\begin{remark}\label{re2}
\emph{%Note that we do not have, a priori, that $\bar J_{[t_0,+\infty)} = \inf_{y\in X} V_{[t_0,+\infty)}(y)$.
Actually we have 
$$\bar J_{[t_0,t_1]} = \inf_{y\in X} V_{[t_0,t_1]}(y).$$
This is obvious because we can split the infimum and write
$$
\bar J_{[t_0,t_1]} = \inf_{y\in X} \inf_{\stackrel{(y(\cdot),u(\cdot))\textrm{ admissible}}{y(t_0)=y}}  \frac{1}{t_1-t_0}C_{[t_0,t_1]}(y(\cdot),u(\cdot)) = \inf_{y\in X} V_{[t_0,t_1]}(y).
$$}
\end{remark}

\medskip
%\paragraph{\underline{Assumptions.}}
In order to state the general turnpike result, we make the following assumptions:

\begin{itemize}
\item[$ (H_1)$.] (Turnpike set) There exists a closed set $\mathcal{T}\subset X$ (called turnpike set) such that
$$
\qquad\forall t_0\in\mathbb R,\quad\forall y\in \mathcal{T},\qquad V_{[t_0,+\infty)}(y)=\bar J_{[t_0,+\infty)}.
$$
\end{itemize}

\begin{itemize}

\item[$ (H_2)$.] (Viability) The turnpike set  $\mathcal{T}$ is \emph{viable}, meaning that, for every $y\in\mathcal{T}$ and for every $t_0\in\mathbb R$, there exists an admissible pair $(y(\cdot),u(\cdot))$ such that $y(t_0)=y$ and $y(t)\in\mathcal{T}$ for every $t\geq t_0$.
Moreover, every admissible trajectory remaining in $\mathcal{T}$ is optimal in the following sense: for every $y\in\mathcal{T}$, for every $t_0\in\mathbb R$, for every admissible pair $(y(\cdot),u(\cdot))$ such that $y(t_0)=y$ and $y(t)\in\mathcal{T}$ for every $t\geq t_0$, we have
$$
V_{[t_0,+\infty)}(y) = \lim_{t\rightarrow+\infty} \frac{1}{t-t_0} C_{[t_0,t]}(y(\cdot),u(\cdot)) .
$$
\end{itemize}

\begin{itemize}

\item[$(H_3)$.] (Controllability) There exist  $\bar \delta_0>0$ and $\bar \delta_1>0$ such that, for every $t_0\in\mathbb R$
and every $t_1\in\mathbb R$ with $t_1>t_0+\bar \delta_0+\bar \delta_1$, and every optimal trajectory
$y(\cdot)$ for the problem $(P_{[t_0,t_1]})$,
\begin{itemize}
\item there exist $\delta_0\in(0,\bar \delta_0]$ and an admissible pair $(y_0(\cdot),u_0(\cdot))$ on $[t_0,t_0+\delta_0]$ such that $y_0(t_0)=y(t_0)$ and $y_0(t_0+\delta_0)\in \mathcal{T}$,

\item for every $y\in\mathcal{T}$, there exist $\delta_1\in(0,\bar \delta_1]$ and  an admissible pair $(y_1(\cdot),u_1(\cdot))$ on $[t_1-\delta_1,t_1]$ such that $y_1(t_1-\delta_1)=y$ and $y_1(t_1)=y(t_1)$.
\end{itemize}

\item[$(H_4)$.] (Coercivity) 
There exist a monotone increasing continuous function $\beta:[0,+\infty)\rightarrow[0,+\infty)$ with $\beta(0)=0$  and a distance 
 $\text{dist} (\cdot,\mathcal T)$  to $\mathcal T$
such that
for every $t_0$ and every $\hat y\in X$,
$$
V_{[t_0,t_1]}(\hat y) \geq \inf_{y\in X}V_{[t_0,t_1]}(y) + \frac{1}{t_1-t_0}\int_{t_0}^{t_1} 
\beta(\text{dist}(\hat y(t),\mathcal{T}))\,dt +\mathrm o(1),$$
holds for any optimal trajectory $\hat y(\cdot)$ starting at $\hat y(t_0)=\hat y$ for the problem $(P_{[t_0,t_1]})$,
where the last term in the above inequality  is an infinitesimal quantity as  $t_1\rightarrow +\infty$.
\end{itemize}
\medskip

Hereafter, we speak of 
\textbf{Assumption (H)} in order to designate assumptions $(H_1)$, $(H_2)$, $(H_3)$ and $(H_4)$.

\medskip

\begin{remark}
\begin{itemize}
\item[(i).] \emph{Under $(H_1)$, we actually have 
$\bar J_{[t_0,+\infty)} = \inf_{y\in X} V_{[t_0,+\infty)}(y)$, $\forall t_0\in \mathbb R$.
}

\item[(ii).] \emph{$(H_2)$ means  that, starting at $y\in\mathcal{T}$, it is better to remain in $\mathcal{T}$ than to leave this set.}

\item[(iii).] \emph{$(H_3)$ is a specific controllability assumption.
For instance, in the case that the initial point $y(t_0)=y_0$ and the final point $y(t_1)=y_1$ in the problem $(P_{[t_0,t_1]})$ are fixed, then $(H_3)$ means that the turnpike set $\mathcal{T}$ is reachable from $y_0$ within time $\bar\delta_0$,  and that $y_1$ is reachable from any point of $\mathcal{T}$ within time $\bar \delta_1$. When the turnpike set $\mathcal T$
is a single point, we refer the reader to \cite{Faulwasser1} for a similar assumption. }

\item[(iv).] \emph{$(H_4)$ is a coercivity  assumption  involving the value function and the turnpike set $\mathcal T$.
It may not be easy to verify this condition. However, under the strict dissipativity property (which will be introduced in the next section),  it is satisfied. We refer the reader to Section \ref{rd} for more discussions about the relationship with the strict dissipativity. 
}
\end{itemize}
\end{remark}
\medskip
We first give a simple example   which satisfies the \textbf{Assumption (H)}.
\begin{example}
Let $\Omega\subset\mathbb R^n$,  $n\geq1$,  be a bounded domain with a smooth boundary $\partial \Omega$, and let $\mathcal D\subset\Omega$ be a non-empty 
 open subset. We denote by $\chi_\mathcal D$ the characteristic function of $\mathcal D$. 
 Let $M>0$ and $y_0\in L^2(\Omega)$ be arbitrarily given.
For  $t_0< t_1$, consider the following optimal control problem for the heat equation: 
$$\;\;\;\;\;\;\;\; \inf\, \frac{1}{t_1-t_0}\int_{t_0}^{t_1}\Big(\|y(\cdot,t)\|_{L^2(\Omega)}^2+\|u(\cdot,t)\|_{L^2(\mathcal D)}^2\Big)\,dt$$
subject to
\begin{equation*}\left\{
\begin{split}
&y_t-\Delta y=\chi_{\mathcal D} u,\;\;\text{in}\;\;\Omega\times(t_0,t_1),\\
&y=0,\;\;\text{on}\;\;\partial\Omega\times(t_0,t_1),\\
&y(\cdot,t_0)=y_0,\;\;y(\cdot,t_1)=0,\;\;\text{in}\;\;\Omega,\\
&\|u(\cdot,t)\|_{L^2(\mathcal D)}\leq M,\;\;\text{for a.e.}\;\;t\in(t_0,t_1).
\end{split}\right.
\end{equation*}
Here, we take $X=L^2(\Omega)$, $U=L^2(\mathcal D)$ and $F=\{u\in U\,\,|\,\, \|u\|_{L^2(\mathcal D)}\leq M\}$. By the standard energy estimate, we can take $E=\{y\in X\,\,|\,\, \|y\|_{L^2(\Omega)}\leq \|y_0\|_{L^2(\Omega)}+M/\lambda_1\}$,
where $\lambda_1>0$ is the first eigenvalue of the Laplace operator with zero Dirichlet boundary condition on $\partial \Omega$.

It is clear that $\bar J_{[t_0,+\infty)}=0$. Let us define the turnpike set $\mathcal T=\{0\}$.   By the $L^\infty$-null controllability and exponential decay of the energy of heat equations, then the above control system with bounded controls is null controllable from each given point $y_0$ within a large time interval (see, e.g., \cite{wang}). Therefore, the assumptions $(H_1)$, $(H_2)$ and $(H_3)$ are satisfied.
Let $y(\cdot)$ be any optimal trajectory starting at $y(\cdot, t_0)=y_0$. 
By the definition of value function, we see that 
$$
V_{[t_0,t_1]}(y_0)\geq \frac{1}{t_1-t_0}\int_{t_0}^{t_1} \|y(\cdot,t)\|^2_{L^2(\Omega)}\,dt.
$$
Hence $(H_4)$ is satisfied with $\beta (r)=r^2$, $r\geq 0$.

\end{example}

\medskip

The main result of this paper is  the following. It says that a general turnpike behavior occurs around the turnpike set $\mathcal T$,
in terms of the time average of the distance from optimal trajectories to $\mathcal T$.

\begin{theorem}\label{thm1}
Assume that $f^0$ is bounded on $\mathbb R\times E\times F$.
\begin{enumerate}
\item[(i).] Under $(H_1)$,  $(H_2)$ and $(H_3)$, for every $t_0\in\mathbb R$ we have
\begin{equation}\label{mi1}
\lim_{t_1\rightarrow +\infty} J_{[t_0,t_1]} = \bar J_{[t_0,+\infty)}.
\end{equation}
\item[(ii).] Further, under the additional assumption  $(H_4)$
 we have 
\begin{equation}\label{guai1}
\lim_{t_1\rightarrow+\infty}\frac{1}{t_1-t_0}\int_{t_0}^{t_1}\beta (\text{dist} (y(t),\mathcal T))\,dt =0,
\end{equation}
for any $t_0$ and any optimal trajectory $y(\cdot)$ of the problem  $(P_{[t_0,t_1]})$.
\end{enumerate}
\end{theorem}

\medskip

\begin{remark}\label{remove}
{\it
The boundedness assumption on $f^0$ in Theorem \ref{thm1} can be removed in case the optimal control problem is autonomous, i.e., when $A$, $f$ and $f^0$ do not depend on $t$, provided that the controllability assumption $(H_3)$ be slightly reinforced, by assuming ``controllability with finite cost": one can steer $y(t_0)$ to the turnpike set $\mathcal T$ within time 
$\delta_0$ and steer any point of $\mathcal T$ to $y(t_1)$ within time $\delta_1$ with a cost that is uniformly bounded with respect to every optimal trajectory $y(\cdot)$ and $y\in \mathcal T$.  For non-autonomous control problems, see also Remark~\ref{dubao1}.
}
\end{remark}

%One can remove the bounded assumption of $f^0$  in Theorem \ref{thm1} if one assumes the control system under consideration is time-invariant, and 
%slightly changes the controllability assumption $(H_3)$ by assuming ``controllability with finite cost", i.e.,  one can steer $y(t_0)$ to the turnpike set $\mathcal T$ within time $(t_0, t_0+\delta_0)$, and also steer each 
%point in $\mathcal T$ to $y(t_1)$ within time $(t_1-\delta_1, t_1)$,
%with a uniformly bounded cost. }

The property \eqref{guai1} is a weak turnpike property, which can be called the $\beta$-integral-turnpike property, and which is even weaker than the measure-turnpike property introduced further in Section \ref{dis}. Indeed, from $\eqref{guai1}$ we infer that for any $\delta>0$, there exists $T_0>t_0$ such that 
\begin{equation*}
\frac{1}{t_1-t_0}\int_{t_0}^{t_1}\beta (\text{dist} (y(t),\mathcal T))\,dt\leq \delta
\end{equation*}
for any $t_1\geq T_0$.
If, for any $\varepsilon>0$,  we set 
\begin{equation*}
Q^\varepsilon_{[t_0,t_1]}=\big\{t\in [t_0,t_1]\ \mid\
 \text{dist} (y(t),\mathcal T)>\varepsilon\big\},\;\;\;\;\forall t_1\geq T_0.
\end{equation*}
Throughout the paper, we denote by $|Q|$ the Lebesgue measure of a subset $Q\subset\mathbb R$.
Then, by Markov's inequality, one can easily derive that
$$
\frac{\left\vert Q^\varepsilon_{[t_0,t_1]}\right\vert}{t_1-t_0}\leq \frac{\delta}{\beta(\varepsilon)},\;\;\;\;\forall t_1\geq T_0.
$$
This is weaker than the property \eqref{bairui1} in Section \ref{dis}.

\begin{proof}[\textbf{Proof of Theorem~\ref{thm1}}]
$(i)$. Let $t_1>t_0+\bar\delta_0+\bar\delta_1$,  with $\bar\delta_0$ and $\bar\delta_1$ as in $(H_3)$.
Let $(y(\cdot),u(\cdot))$ be an optimal pair for the problem $(P_{[t_0,t_1]})$.
By $(H_2)$ and $(H_3)$, there exist $\delta_0\in(0,\bar\delta_0]$, $\delta_1\in(0,\bar\delta_1]$ and an admissible pair $(\widetilde y(\cdot),\widetilde u(\cdot))$ such that
\begin{itemize}
\item $\widetilde y(\cdot)$ steers the control system from $y(t_0)$ to $\mathcal{T}$ within the time interval $[t_0,t_0+\delta_0]$,
\item  $\widetilde y(\cdot)$ remains in  $\mathcal{T}$ within the time interval $[t_0+\delta_0,t_1-\delta_1]$,
\item  $\widetilde y(\cdot)$ steers the control system from $\widetilde y(t_1-\delta_1)\in\mathcal{T}$ to $y(t_1)$ within the time interval  $[t_1-\delta_1,t_1]$.
\end{itemize}
These trajectories are drawn in Figure \ref{fig1}.

\begin{figure}[h]
\centering 
\includegraphics[width=10cm]{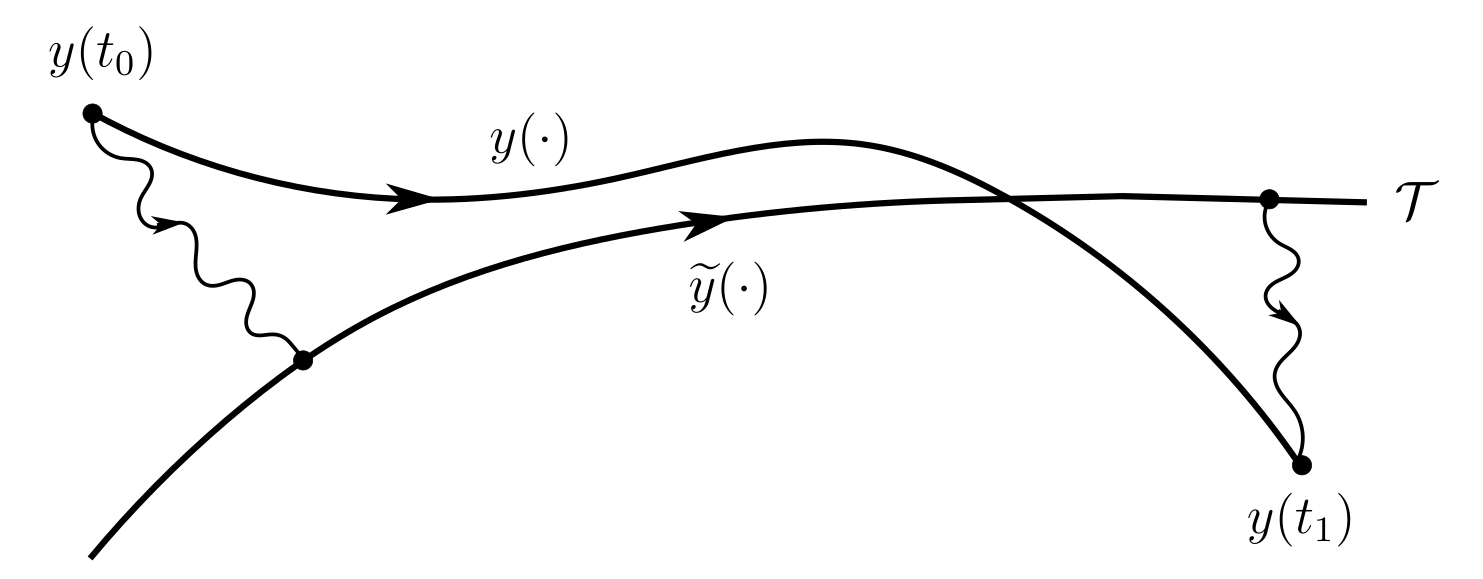}
\caption{Optimal trajectory $y(\cdot)$, and admissible trajectory $\widetilde y(\cdot)$ remaining along the turnpike set $\mathcal{T}$ as long as possible.} 
\label{fig1}
\end{figure}

Its cost of time average within the time interval $[t_0,t_1]$ is
\begin{multline}\label{sum}
\frac{1}{t_1-t_0} C_{[t_0,t_1]}(\widetilde y(\cdot),\widetilde u(\cdot))
= \frac{1}{t_1-t_0} C_{[t_0,t_0+\delta_0]}(\widetilde y(\cdot),\widetilde u(\cdot)) + \frac{1}{t_1-t_0} C_{[t_1-\delta_1,t_1]}(\widetilde y(\cdot),\widetilde u(\cdot)) \\
+ \frac{1}{t_1-t_0} C_{[t_0+\delta_0,t_1-\delta_1]}(\widetilde y(\cdot),\widetilde u(\cdot)).
\end{multline}
Since $f^0$ is bounded on $\mathbb R\times E\times F$, the first two  terms on the right hand side of  \eqref{sum} converge to zero as $t_1\rightarrow +\infty$.
Since $\widetilde y(t_0+\delta_0)\in\mathcal{T}$, by $(H_2)$ we have
\begin{equation}\label{tj3}
V_{[t_0+\delta_0,+\infty)}(\widetilde y(t_0+\delta_0))=
\lim_{t_1\rightarrow+\infty}\frac{1}{t_1-\delta_1-(t_0+\delta_0)}C_{[t_0+\delta_0,t_1-\delta_1]}(\widetilde y(\cdot),\widetilde u(\cdot) ).
\end{equation}
As $\widetilde y(t_0+\delta_0)\in\mathcal{T}$, by $(H_1)$ we infer
\begin{equation}\label{tj2}
V_{[t_0+\delta_0,+\infty)}(\widetilde y(t_0+\delta_0))=\bar J_{[t_0+\delta_0,+\infty)}.
\end{equation}
We now claim that 
\begin{equation}\label{tianjing1}
\bar J_{[t_0+\delta_0,+\infty)}\leq \bar J_{[t_0,+\infty)}.
\end{equation}
We postpone the proof of this claim and first see how it could be used in showing the convergence \eqref{mi1}.
Therefore, we derive from \eqref{tj2} and \eqref{tianjing1}  that 
\begin{equation*}
V_{[t_0+\delta_0,+\infty)}(\widetilde y(t_0+\delta_0))\leq \bar J_{[t_0,+\infty)}.
\end{equation*}
This, together with \eqref{sum} and \eqref{tj3}, indicate that 
\begin{equation}\label{hui1}
\lim_{t_1\rightarrow+\infty}\frac{1}{t_1-t_0} C_{[t_0,t_1]}(\widetilde y(\cdot),\widetilde u(\cdot))\leq \bar J_{[t_0,+\infty)}.
\end{equation}

On the other hand, by the construction above, $(\widetilde y(\cdot),\widetilde u(\cdot))$ is an admissible pair satisfying the terminal state constraint $R(t_0,\widetilde y(t_0),t_1,\widetilde y(t_1))=0$, we have
 $$
 J_{[t_0,t_1]} \leq \frac{1}{t_1-t_0} C_{[t_0,t_1]}(\widetilde y(\cdot),\widetilde u(\cdot)).
 $$
This, combined with \eqref{hui1}, infers that
   $$\limsup_{t_1\rightarrow+\infty} J_{[t_0,t_1]} \leq \bar J_{[t_0,+\infty)}.$$
Which, along with  \eqref{lower limit}, leads to \eqref{mi1}.

Next, we present the proof of the claim \eqref{tianjing1}. Let $(\bar y(\cdot),\bar u(\cdot))$ be an optimal pair 
for the problem $(\bar P_{[t_0,t_1]})$. Then
\begin{multline*}
\bar J_{[t_0,t_1]}-\bar J_{[t_0+\delta_0,t_1]}
=
\frac{1}{t_1-t_0}\int_{t_0}^{t_0+\delta_0}f^0(t,\bar y(t),\bar u(t))\,dt+\frac{1}{t_1-t_0}
\int_{t_0+\delta_0}^{t_1}f^0(t,\bar y(t),\bar u(t))\,dt-\bar J_{[t_0+\delta_0,t_1]}\\
=\frac{1}{t_1-t_0}\int_{t_0}^{t_0+\delta_0}f^0(t,\bar y(t),\bar u(t))\,dt+\Big(\frac{t_1-t_0-\delta_0}{t_1-t_0}-1\Big)\times
\frac{1}{t_1-t_0-\delta_0}\int_{t_0+\delta_0}^{t_1}f^0(t,\bar y(t),\bar u(t))dt\\
+
\frac{1}{t_1-t_0-\delta_0}\int_{t_0+\delta_0}^{t_1}f^0(t,\bar y(t),\bar u(t))dt-\bar J_{[t_0+\delta_0,t_1]}.
\end{multline*}
Since $(\bar y(\cdot),\bar u(\cdot))$ is also admissible for the problem $(\bar P_{[t_0+\delta_0,t_1]})$,
\begin{equation*}
\frac{1}{t_1-t_0-\delta_0}\int_{t_0+\delta_0}^{t_1}f^0(t,\bar y(t),\bar u(t))dt\geq\bar J_{[t_0+\delta_0,t_1]},
\end{equation*}
we see that 
\begin{multline*}
\bar J_{[t_0,t_1]}-\bar J_{[t_0+\delta_0,t_1]}\\ \geq
\frac{1}{t_1-t_0}\int_{t_0}^{t_0+\delta_0}f^0(t,\bar y(t),\bar u(t))\,dt+\Big(\frac{t_1-t_0-\delta_0}{t_1-t_0}-1\Big)\times
\frac{1}{t_1-t_0-\delta_0}\int_{t_0+\delta_0}^{t_1}f^0(t,\bar y(t),\bar u(t))dt.\\
\end{multline*}
By the boundedness of $f^0$ on $\mathbb R\times E\times F$ (i.e., there exists $M>0$ such that $|f^0(\cdot)|\leq M$), we obtain
\begin{equation*}
\bar J_{[t_0,t_1]}-\bar J_{[t_0+\delta_0,t_1]} \geq -\frac{2M\delta_0}{t_1-t_0},
\end{equation*}
which implies \eqref{tianjing1} as $t_1\rightarrow +\infty$.

\medskip
$(ii)$. By the definition of the value function $V_{[t_0,t_1]}(\cdot)$, we obtain
\begin{equation}\label{jican1}
 V_{[t_0,t_1]}(y(0))\leq\frac{1}{t_1-t_0} C_{[t_0,t_1]}(y(\cdot),u(\cdot)).
\end{equation}
By $(H_4)$ and Remark~\ref{re2},  we have
\begin{equation}\label{jican2}
V_{[t_0,t_1]}(y(0))
\geq \bar J_{[t_0,t_1]}+\frac{1}{t_1-t_0}\int_{t_0}^{t_1} 
\beta(\text{dist}(y(t),\mathcal{T}))\,dt +\mathrm o(1),
\end{equation}
as $t_1\rightarrow +\infty$.
By \eqref{mi1} we infer  
$$
\lim_{t_1\rightarrow+\infty}\frac{1}{t_1-t_0} C_{[t_0,t_1]}(y(\cdot),u(\cdot))=\lim_{t_1\rightarrow+\infty}J_{[t_0,t_1]}=\bar J_{[t_0,+\infty)}.$$
This, together with \eqref{jican1} and \eqref{jican2}, indicates 
$$
\limsup_{t_1\rightarrow+\infty}\frac{1}{t_1-t_0}\int_{t_0}^{t_1} 
\beta(\text{dist}(y(t),\mathcal{T}))\,dt =0,
$$
which completes the proof.
\end{proof}

\begin{remark}\label{nonc}
{\it 
In the proof of Theorem \ref{thm1}, the role of controllability assumption $(H_3)$ is to ensure that there is  an admissible trajectory $\widetilde y(\cdot)$ satisfying the terminal state condition $R(t_0,\widetilde y(t_0), t_1,\widetilde y(t_1))=0$ and with a comparable cost (i.e., \eqref{hui1}).  

Note that $(H_3)$ can be weakened to some cases where controllability may fail:
take any control system that is asymptotically controllable to the turnpike set $\mathcal{T}$. This is the case for the heat equation which is asymptotically controllable for any given point (cf., e.g., \cite[Chapter 7]{LiXunjing}). Then, if one waits for a certain time, one will arrive at some neighborhood of $\mathcal{T}$.
Similarly, to run the proof as in Theorem \ref{thm1},  one needs an assumption which is stronger than $(H_2)$. More precisely,  one needs viability, not only along $\mathcal{T}$, but also in a neighborhood of $\mathcal{T}$.
Under these assumptions, we believe that one can design a turnpike result for this control system with free final point. 

In any case, note that, when the final point is free, having a turnpike property is more or less equivalent to having an asymptotic stabilization to $ \mathcal{T}$ (see also an analogous discussion in \cite[Remark 2]{Faulwasser1}).
    If additionally one wants to fix the final point, then one would need the existence of a trajectory steering any point of the neighborhood of $\mathcal{T}$ to the final point.}
\end{remark}

\begin{remark}\label{dubao1}
{\it
As seen in the proof of Theorem \ref{thm1}, the assumption of boundedness of $f^0$ is used two times: the first one, in order to bound the first two terms of  \eqref{sum};  the second one, in order to prove \eqref{tianjing1}. For autonomous optimal control problems, on the one part we have  $\bar J_{[t_0+\delta_0,+\infty)}= \bar J_{[t_0,+\infty)}$ (see Remark \ref{inv}) and then \eqref{tianjing1} is true, and on the other part the first two terms at the right-hand side of \eqref{sum} converge to zero as
$t_1\rightarrow +\infty$ under the ``controllability with finite cost" assumption mentioned in Remark \ref{remove}.  In contrast, for non-autonomous optimal control problems the situation may be more complicated, in particular due to the dependence on time of $f^0$. The assumption of boundedness of $f^0$ is quite strong and could of course be weakened in a number of ways so as to ensure that the above proof still works. We prefer keeping this rather strong assumption in order to put light in the main line of the argument, not going into too technical details. Variants are easy to derive according to the context.
}
\end{remark}
 
% Continued with Remark \ref{remove}. As seen in the proof  of Theorem \ref{thm1}, 
%there are only two places where the boundedness of $f^0$ is used: one is the first two terms in  \eqref{sum};
%and another one is the proof of claim \eqref{tianjing1}.
%In fact, if the control system is time-invariant, 
%then it follows from Remark \ref{inv} that $\bar J_{[t_0+\delta_0,+\infty)}= \bar J_{[t_0,+\infty)}$, which is stronger than \eqref{tianjing1}. While the first two terms on the right hand side of  \eqref{sum} converge to zero as $t_1\rightarrow +\infty$, under  ``controllability with finite cost" assumption mentioned in Remark \ref{remove}.

\section{Relationship with (strict) dissipativity}\label{sec_turnpike}\label{rd}

In this section, we make precisely  the relationship between the strict dissipativity property (which we recall in Section \ref{dd})
and the so-called measure-turnpike property (which we define in Section \ref{dis}). 

\subsection{What is (strict) dissipativity}\label{dd}
To fix ideas, in this section we only consider the autonomous case.
Let $X$, $U$, $E$ and $F$ be the same as in Section \ref{general}. Let $A(\cdot)\equiv A$
generate a $C_0$ semigroup $\{e^{tA}: t\geq 0\}$ on $X$, and let $f$ and $f^0$ be time-independent.
To simplify the notation,
for every $T>0$, we here consider the optimal control problem
\begin{equation*}
\boxed{
(\bar P_{[0,T]})\qquad \left\{\begin{split}
& \inf J^T(y(\cdot),u(\cdot))=\frac{1}{T}\int_0^Tf^0(y(t),u(t))\,dt,\\
&\text{subject to}\;\;\; \dot{y}(t)=Ay(t)+f(y(t),u(t)), \quad t\in[0,T], \\[1mm]
& y(t)\in E, \quad  u(t)\in F,\qquad t\in[0,T]. \\[1mm]
\end{split}\right.}
\end{equation*}
Indeed, the above problem $(\bar P_{[0,T]})$ coincides with  $(\bar P_{[t_0,t_1]})$ in Section~\ref{general} for $t_0=0$ and $t_1=T$. 
Note that the terminal states $y(0)$ and $y(T)$ are left free in the problem $(\bar P_{[0,T]})$. 
Recall that the solutions $(y(\cdot),u(\cdot))\in C([0,T];X)\times L^2(0,T;U)$ are considered in the mild sense, meaning that
$$y(\tau)=e^{\tau A}y(0)+\int_{0}^\tau e^{(\tau-t)A} f(y(t),u(t))\,dt,\qquad\forall\tau\in[0,T],$$
or equivalently,
\begin{equation*}\label{5311}
\begin{split}
\langle \varphi, y(\tau)\rangle_{X^*,X}-\langle \varphi, y(0)\rangle_{X^*,X}=\int_{0}^\tau \Big( \langle A^*\varphi, y(t)\rangle_{X^*,X}+\langle  \varphi,f(y(t),u(t))\rangle_{X^*,X}\Big) dt ,
\end{split}
\end{equation*}
for each $\tau\in[0,T]$ and $\varphi\in D(A^*)$, where $A^*:D(A^*)\subset X^*\rightarrow X^*$ is the adjoint operator of $A$, and  $\langle\cdot,\cdot\rangle_{X^*,X}$ is the dual paring 
between $X$ and its dual space $X^*$.

Likewise,
we say $(y(\cdot), u(\cdot))$ an admissible pair to the problem $(\bar P_{[0,T]})$ if it satisfies the state
equation and the above state-control constraint. 
Assume that, for any $T>0$,  $(\bar P_{[0,T]})$ has at least one optimal solution denoted by $(y^T(\cdot),u^T(\cdot))$, and we set
$$\bar J^T=J^T(y^T(\cdot),u^T(\cdot)).$$
Note that $\bar J^T$ does not depend on the optimal solution under consideration.

In the finite-dimensional case where $X = \mathbb R^n$ and $U=\mathbb R^m$, without loss of generality, we may take $A = 0$, and then the control system is $\dot{y}(t)=f(y(t),u(t))$.
We refer the reader to \cite{Faulwasser1,TZ1} for the asymptotic behavior of optimal solutions of such optimal control problems with constraints on the terminal states.

Consider the static optimal control problem
\begin{equation*}\boxed{
(P_s)\qquad \left\{\begin{split}
& \inf J_s(y,u)= f^0(y,u), \\
& \text{subject to}\;\;\; Ay+f(y,u)=0, \\
& y\in E, \quad  u\in F, \\
\end{split}\right.}
\end{equation*}
where the first equation means that
$$
\langle A^*\varphi, y\rangle_{X^*,X}+\langle \varphi,f(y,u)\rangle_{X^*,X}=0,\qquad\forall\varphi\in D(A^*). 
$$
As above, we assume that there exists at least one optimal solution $(y_s,u_s)$ of $(P_s)$. Such existence results are as well standard, for instance in the case where $A$ is an elliptic differential operator (see \cite[Chapter 3, Theorem 6.4]{LiXunjing}).
We set 
$$\bar J_s=J_s(y_s,u_s).$$
Note that $\bar J_s$ does not depend on the optimal solution that is considered.
Of course, uniqueness of the minimizer cannot be ensured in general because the problem is not assumed to be convex.
Note that $(y_s,u_s)$ is admissible for the problem $(\bar P_{[0,T]})$ for any $T>0$, meaning that it satisfies the constraints and is a solution of the control system.

We next define the notion of dissipativity for the infinite-dimensional controlled system, which is originally due to \cite{Willems} for finite-dimensional dynamics (see also related definitions in  \cite{Faulwasser1}). Recall that the continuous function $\alpha: [0,+\infty)\rightarrow [0,+\infty)$ with $\alpha(0)=0$ is said to be a $\mathcal K$-class function if it is monotone increasing.

\begin{definition}%\label{steady-dis}
We say that $\{(\bar P_{[0,T]})\, \mid\, T>0\}$ is \emph{dissipative} at an optimal stationary point $(y_s, u_s)$ with respect to the \emph{supply rate function}
\begin{equation}\label{5234}
\omega(y,u)= f^0(y,u)-f^0(y_s,  u_s),\qquad \forall(y,u)\in E\times F,
\end{equation}
if there exists a \emph{storage function} $S:E\rightarrow\mathbb R$, locally bounded and bounded from below, such that, for any $T>0$, the dissipation inequality
\begin{equation}\label{5224}
S(y(0))+\int_0^\tau\omega(y(t),u(t))\,dt\geq S(y(\tau)),\;\;\forall \tau\in[0,T],
\end{equation}
holds true,  for any admissible pair $(y(\cdot),u(\cdot))$. 

We say it is \emph{strictly dissipative} at  $(y_s,  u_s)$
with respect to the supply rate function $\omega$ if there exists a $\mathcal{K}$-class function $\alpha(\cdot)$ such that, for any $T>0$, the strict dissipation inequality
\begin{equation}\label{5225}
S(y(0))+\int_0^\tau\omega(y(t),u(t))\,dt\geq S(y(\tau))+\int^\tau_0\alpha\big(\|(y(t)- y_s, u(t)- u_s)\|_{X\times U}\big)\,dt,
\;\;\forall \tau\in[0,T],
\end{equation}
holds true, for any admissible pair $(y(\cdot),u(\cdot))$.  
The function $d(\cdot)= \alpha(\|(y(\cdot)- y_s, u(\cdot)- u_s)\|_{X\times U}) $ in \eqref{5225} is  called the dissipation rate.
\end{definition}

Although there are many possibly different notions of dissipativity introduced in the literature (such as the positivity or the local boundedness of the storage function in their definitions, cf., e.g., \cite[Chapter 4]{disbook}), they are proved to be equivalent in principle between with each other.
Note that a storage function is defined up to an additive constant.
We here define the storage function $S : E \rightarrow \mathbb R$ to take real values instead of positive real values. Since $S$ is assumed to be bounded from below, one could as well consider $S : E \rightarrow  \mathbb [0,+\infty)$.
We mention that no regularity is a priori required to define it. 
Actually, storage functions do possess some regularity properties, such as $C^0$ or $C^1$ regularity, under suitable assumptions.
For example, the controllable and observable systems with  positive transfer functions are dissipative with quadratic storage functions (see \cite[Section 4.4.5]{disbook} for instance).

When a system is dissipative with a given supply rate function, the question of finding a storage function has been extensively studied. This question is closely similar to the problem of finding a suitable Lyapunov function in the Lyapunov second method ensuring the stability of a system.
 For linear systems with a quadratic
supply rate function, the existence of a storage function boils down to solve a Riccati inequality.
In general, storage functions are closely related to viscosity solutions of a partial differential inequality, called a \emph{Hamilton-Jacobi} inequality. We refer the  reader to \cite[Chapter 4]{disbook} for more details on this subject.

An equivalent characterization of the dissipativity in \cite{Willems}
can be described by the so-called \emph{available storage}, which is defined as 
$$
S_a(y)\triangleq\sup_{t\geq0,\,(y(\cdot),u(\cdot))}\Big\{-\int_0^t \omega(y(\tau),u(\tau))\,d\tau\Big\},
$$
where the $\sup$ is taken over all  admissible pairs $(y(\cdot),u(\cdot))$ (meaning that satisfy the dynamic controlled system and state-control constraints) with initial value $y(0)=y$.
In fact, for every $y\in E$, $S_a(y)$ can be seen as the maximum amount of ``energy'' which can be extracted from the system with initial state $y=y(0)$.
It has been shown by Willems \cite{Willems} that
the problem $(\bar P_{[0,T]})$ is dissipative at  $(y_s, u_s)$ with respect to the supply rate function $\omega(\cdot,\cdot)$
 if and only if  $S_a(y)$ is finite for every $y\in E$.
 
We provide a specific example of a (strictly) dissipative control system.

 \begin{example}
 Let $\Omega\subset\mathbb R^n$ ($n\geq1$) be a smooth and  bounded domain, and let $\mathcal D\subset\Omega$ be a non-empty 
 open subset. Denote by $\langle\cdot,\cdot \rangle$ and $\|\cdot\|$ the inner product and norm in $L^2(\Omega)$
 respectively. For each $T>0$, consider the optimal control problem 
$$ \;\;\;\;\;\;\inf \,\int_0^T \Big(\langle y(t), \chi_\mathcal D u(t)\rangle+\|u(t)\|^2\Big)\,dt,$$
subject to
\begin{equation*}\left\{
\begin{split}
&y_t-\Delta y=\chi_\mathcal D u,\;\;\text{in}\;\;\Omega\times(0,T),\\
&y=0,\;\;\text{on}\;\;\partial\Omega\times(0,T),\\
&\|y(t)\|\leq 1, \;\; \|u(t)\|\leq 1,  \;\;\;\forall t\in[0,T].
\end{split}\right.
\end{equation*}
Notice that the corresponding static problem has a unique solution $(0,0)$.
We show that this problem is strictly dissipative at $(0,0)$ with respect to the supply rate
$$\omega (y,u)=\langle y,\chi_\mathcal D u\rangle +\|u\|^2,\;\;\forall (y, u)\in L^2(\Omega)\times L^2(\Omega).$$
In fact,  integrating the heat equation by parts leads to
\begin{equation*}
\int_0^\tau\langle y(t),\chi_\mathcal D u(t) \rangle \,dt
=\frac{\|y(\tau)\|^2-\| y(0)\|^2}{2}+\int_0^\tau\|\nabla y(t)\|^2 dt,
\end{equation*}
 for any $\tau\in[0,T]$. This, together with the definition of $\omega(\cdot,\cdot)$ above  and the Poincar\'e inequality, indicates that the strict
dissipation inequality
$$
S(y(\tau))+c\int_0^\tau\Big(\|y(t)\|^2+\|u(t)\|^2\Big)\,dt\leq 
S(y(0))+\int_0^\tau \omega(y(t),u(t))\,dt,\;\;\forall \tau\in[0,T],
$$ 
holds with $\alpha(\gamma)=c\gamma^2$ for some constant $c>0$, and a storage function $S(\cdot)$ given by
$$S(y)=\frac{1}{2}\|y\|^2, \;\;\forall y\in L^2(\Omega).$$ 
Thus,  this problem has the strict dissipativity property at $(0,0)$.
\end{example}

\subsection{Strict dissipativity implies measure-turnpike}\label{dis}

Next, we introduce a rigorous  definition of measure-turnpike for optimal control problems. 

\begin{definition}%\label{turnpike}
We say that  $\{(\bar P_{[0,T]})\, \mid\, T>0\}$ enjoys the \emph{measure-turnpike property} at $(y_s,u_s)$ if, for every $\varepsilon>0$, there exists $\Lambda(\varepsilon)>0$ such that
\begin{equation}\label{bairui1}
\mathcal |Q_{\varepsilon,T}|\leq \Lambda(\varepsilon),\qquad\forall T>0,
\end{equation}
where
\begin{equation}\label{5227}
Q_{\varepsilon,T}=\left\{ t\in[0,T]\, \mid\, \left\|(y^T(t)-y_s,u^T(t)-u_s)\right\|_{X\times U}> \varepsilon \right\}.
\end{equation}

\end{definition}

We refer the reader to \cite{CarlsonBOOK, Faulwasser1, Z2} (and references therein) for similar definitions.
In this definition, the set $Q_{\varepsilon,T}$ measures the set of times at which the optimal trajectory and control stay outside an $\varepsilon$-neighborhood of $(y_s,u_s)$ for the strong topology. 
We stress that the measure-turnpike property defined above concerns both state and control. In the existing literature (see, e.g., \cite{CarlsonBOOK}), the turnpike phenomenon is often studied only for the state, meaning that, for each $\varepsilon>0$, the (Lebesgue) measure of the set $\{t\in[0,T]\, \mid\, \|y^T(t)-y_s\|_X>\varepsilon\}$ is uniformly bounded  for any $T>0$.

In the following result, we establish the measure-turnpike property for optimal solutions of $(\bar P_{[0,T]})$ (as well for $(\bar P_{[t_0,t_1]})$) under the strict dissipativity assumption, as the parameter $T$ goes to infinity. This implies that \emph{any} optimal solution $(y^T(\cdot),u^T(\cdot))$ of $(\bar P_{[0,T]})$ remains \emph{essentially} close to \emph{some} optimal solution $(y_s,u_s)$ of $(P_s)$.  Our results can be seen in the stream of the recent works \cite{Faulwasser1, PZ1, PZ2, TZ1}.

\begin{theorem}\label{turnpikeproperty1}
Let $E$ be a bounded subset of $X$. 
\begin{enumerate}
\item[(i).] If  $\{(\bar P_{[0,T]})\, \mid\, T>0\}$ is dissipative at $(y_s,u_s)$ with respect to the supply rate function $\omega(\cdot,\cdot)$ given by \eqref{5234}, then
\begin{equation}\label{5226}
\bar J^T=\bar J_s + O(1/T),
\end{equation}
as $T\rightarrow+\infty$.
\item[(ii).] If  $\{( \bar P_{[0,T]})\, \mid\, T>0\}$ is strictly dissipative at $(y_s,u_s)$ with respect to the supply rate function $\omega(\cdot,\cdot)$ given by \eqref{5234}, then it satisfies the measure-turnpike property at $(y_s,u_s)$.
\end{enumerate}
\end{theorem}
\medskip
\begin{remark}
\emph{
Note that $(\bar P_{[0,T]})$ is defined without any constraints on the terminal states. However,
under appropriate controllability assumptions (similar to $(H_3)$ in Section \ref{general}),
one can also treat the case of terminal state constraint $R(\cdot)=0$ (see also  \cite{Faulwasser1} and \cite{Grune2}).}
\end{remark}

\begin{remark}
\emph{
From Theorem \ref{turnpikeproperty1},  we see that  strict dissipativity is sufficient for the measure-turnpike property for the optimal control problem. This fact was observed in the previous works \cite{Grune1, Faulwasser1, Grune2}. For the converse statements, i.e., results which show that the turnpike property implies  strict dissipativity, we refer the reader to \cite{G3} and \cite{Faulwasser1}.  In \cite{G3}, the authors first defined a turnpike-like behavior concerning all trajectories whose associate cost is close to the optimal one. This behavior is stronger than the measure-turnpike property,  which only concerns the optimal  trajectories.  Then, the implication ``\,turnpike-like behavior $\Rightarrow$ strict dissipativity''  was proved in \cite{G3}.  Besides, the implication  ``\,exact turnpike property $\Rightarrow$ strict dissipativity along optimal trajectories'' was shown 
in \cite{Faulwasser1}, where the exact turnpike property means that the optimal solutions have to remain exactly at an optimal steady-state for most part of the long-time horizon.
}

\end{remark}

%\medskip

\begin{proof}[\textbf{Proof of Theorem~\ref{turnpikeproperty1}}.]
We first prove the second point of the theorem.
Let $T>0$ and let $(y^T(\cdot),u^T(\cdot))$ be any optimal solution of the problem $(\bar P_{[0,T]})$.  By the strict dissipation inequality \eqref{5225} applied to $(y^T(\cdot),u^T(\cdot))$, we have
\begin{equation}\label{5123}
\frac{1}{T}\int_0^T\alpha\big(\|(y^T(t)-y_s,u^T(t)-u_s)\|_{X\times U}\big)\,dt\leq
\bar J^T-\bar J_s+\frac{S(y^T(0))-S(y^T(T))}{T}.
\end{equation}
Note that $\alpha\big(\|(y^T(t)-y_s,u^T(t)-u_s)\|_{X\times U}\big)\geq \alpha(\varepsilon)$ whenever $t\in Q_{\varepsilon,T}$, where $Q_{\varepsilon,T}$ is defined by \eqref{5227}.
Since $E\subset X$ is a bounded subset and $S(\cdot)$ is locally bounded, there exists $M>0$ such that $|S(y)|\leq M$ for every $y\in E$.
Therefore, it follows from \eqref{5123} that
\begin{equation}\label{5132}
\frac{|Q_{\varepsilon,T}|}{T}\leq \frac{1}{\alpha(\varepsilon)}\left( \bar J^T-\bar J_s+\frac{2M}{T}\right).
\end{equation}
On the other hand, noting that $(y_s,u_s)$ is admissible for $(\bar P_{[0,T]})$ for any $T>0$, we have
\begin{equation}\label{5133}
\bar J^T\leq \frac{1}{T}\int_0^Tf^0(y_s,u_s)\,dt=f^0(y_s,u_s)=\bar J_s.
\end{equation}
This, combined with \eqref{5132}, leads to $|Q_{\varepsilon,T}|\leq \frac{2M}{\alpha(\varepsilon)}$ for every $T>0$. The second point of the theorem follows.

Let us now prove the first point.
On the one hand, it follows from \eqref{5133} that 
\begin{equation*}\label{5121}
\limsup_{T\rightarrow \infty} \bar J^T\leq \bar J_s.
\end{equation*}
By the dissipation inequality \eqref{5224} applied to any optimal solution $(y^T(\cdot),u^T(\cdot))$ of $(\bar P_{[0,T]})$, we get
\begin{equation*}
S(y^T(0))+\int_0^Tf^0(y^T(t),u^T(t))\,dt\geq
Tf^0(y_s,u_s)+S(y^T(T)),
\end{equation*}
which leads to
\begin{equation*}
\bar J_s\leq \bar J^T+\frac{S(y^T(0))-S(y^T(T))}{T}.
\end{equation*}
Since $E$ is a bounded subset in $X$ and since the storage function $S(\cdot)$ is locally bounded and bounded below, we infer that
\begin{equation*}
\bar J_s\leq \liminf_{T\rightarrow\infty}\bar J^T.
\end{equation*}
Then \eqref{5226} follows.
\end{proof}

\begin{remark}
\emph{
The above proof borrows ideas from \cite[Theorem 5.3]{Grune2} and \cite{Faulwasser1}. We used in a crucial way the fact that any solution of the steady-state problem $(P_s)$ is admissible for the problem $(\bar P_{[0,T]})$ under consideration. This is due to the fact that  the terminal states are let free in $(\bar P_{[0,T]})$.  Note that we only use the boundedness of $y^T(0)$ in the proof.}
\end{remark}

\subsection{Dissipativity and Assumption (H)}

Under the (strict) dissipativity property,  we can verify the abstract \textbf{Assumption (H)} for the autonomous case in Section~\ref{general}.

\begin{proposition}\label{thm3}
Assume that, for any $t_0$ and $t_1$, the problem $(\bar P_{[t_0,t_1]})$ is dissipative at $(y_s,u_s)$ with the supply rate $\omega(y,u)=f^0(y,u)-f^0(y_s,u_s)$, and the associated storage function $S(\cdot)$  is bounded on $E$. Then
\begin{itemize}
\item[(i).] $\bar J_{[t_0,+\infty)} = \bar J_s$, $\forall t_0\in \mathbb R$.
\item[(ii).] There exists a turnpike set $\mathcal{T} = \{ y_s \}$ such that 
$(H_1)$ and
 $(H_2)$ are satisfied.
\item[(iii).] Moreover, if  $(H_3)$ is satisfied and it is strictly dissipative at $(y_s,u_s)$ with dissipation rate $d(\cdot)=\alpha(\|y(\cdot)-y_s\|_X)$, then $(H_4)$ is satisfied with $\beta(\cdot)=\alpha(\cdot)$ and $\text{dist}\,(y,\mathcal T)=\|y-y_s\|_X$.
 \end{itemize}

\end{proposition}

\begin{proof}

$(i)$. With a slight modification, the proof is the same as that of the first point of
Theorem~\ref{turnpikeproperty1}.

$(ii)$. Since $(y_s,u_s)$ is an equilibrium point, the constant pair $(y_s,u_s)$ is admissible on any time interval.
By the definition,
$$
\bar J_{[t_0,t_1]}\leq V_{[t_0,t_1]}(y_s)\leq \frac{1}{t_1-t_0}\int_{t_0}^{t_1}f^0(y_s,u_s)\,dt=\bar J_s.
$$
This, along with $(i)$, indicates that 
$$
V_{[t_0,+\infty)}(y_s)=\lim_{t_1\rightarrow+\infty}V_{[t_0,t_1]}(y_s)= \bar J_s.
$$
Hence, the assumptions $(H_1)$ and $(H_2)$ hold.

$(iii).$
Let $(\tilde y(\cdot),\tilde u(\cdot))$  be an optimal solution to the problem $(P_{[t_0,t_1]})$.
Then, by the strict dissipativity property we have
$$
S(\tilde y(t_1))+ \int_{t_0}^{t_1} \alpha(\|\tilde y(t)-y_s\|_X)\, dt \leq S(\tilde y(t_0))+\int_{t_0}^{t_1} \big(f^0(\tilde y(t),\tilde u(t))-f^0(y_s,u_s)\big)\, dt.
$$
Which is equivalent to 
\begin{equation}\label{xie1v}
J_{[t_0,t_1]}\geq \bar J_s+ \frac{1}{t_1-t_0}\int_{t_0}^{t_1}\alpha(\|\tilde y(t)-y_s\|_X)\, dt
+ \frac{S(\tilde y(t_1))-S(\tilde y(t_0))}{t_1-t_0}.
\end{equation}
Because
\begin{equation*}
\begin{split}
J_{[t_0,t_1]}&=\bar J_{[t_0,t_1]}+(J_{[t_0,t_1]}-\bar J_{[t_0,t_1]})\\
&=\inf_{y\in X}V_{[t_0,t_1]}(y)+\big(J_{[t_0,t_1]}-\inf_{y\in X}V_{[t_0,t_1]}(y)\big)\\
&\leq V_{[t_0,t_1]}(\tilde y(t_0))+\big(J_{[t_0,t_1]}-\inf_{y\in X}V_{[t_0,t_1]}(y)\big).
\end{split}
\end{equation*}
The last inequality, along with \eqref{xie1v}, indicates that
\begin{multline}\label{xie2}
V_{[t_0,t_1]}(\tilde y(t_0))\geq
\inf_{y\in X}V_{[t_0,t_1]}(y)
+ \frac{1}{t_1-t_0}\int_{t_0}^{t_1}\alpha(\|\tilde y(t)-y_s\|_X)\, dt\\
+ \frac{S(\tilde y(t_1))-S(\tilde y(t_0))}{t_1-t_0}
+ \bar J_s-J_{[t_0,t_1]}.
\end{multline}
As the storage function $S(\cdot)$ is bounded on $E$,
by $(i)$ and \eqref{mi1} we infer 
$$\lim_{t_1\rightarrow+\infty}J_{[t_0,t_1]}=\bar J_s,$$ 
Hence, the sum of last three terms in \eqref{xie2} is an infinitesimal 
quantity as $t_1\rightarrow +\infty$, and thus $(H_4)$ holds.

\end{proof}

\begin{remark}\emph{
Proposition \ref{thm3} explains  the role of dissipativity in the general turnpike phenomenon. 
It reflects that dissipativity allows one to identify the limit value $\bar J_{[t_0,+\infty)}$, that dissipativity implies $(H_1)$ and $(H_2)$, and that strict dissipativity, plus $(H_3)$, implies $(H_4)$. 
Recall that $(H_3)$ is a controllability assumption.}
\end{remark}

%%%%%%%%%%%%%%%%%%%%%%%%%%%%%%%%%%%%%%%%%%%%%%%%%%%%%%%%%
\subsection{Some comments on the periodic turnpike phenomenon}
Inspired from \cite{Willems} and \cite{Zon}, we  introduce the  concept of (strict) dissipativity with respect to  a periodic trajectory. Let $A(\cdot)$, $f(\cdot)$ and $f^0(\cdot)$ be periodic in time with a period $\Pi>0$. 

\begin{definition}\label{perdis}
We say the problem $(\bar P_{[t_0,t_1]})$ is dissipative with respect to a $\Pi$-periodic trajectory $(\hat y(\cdot),\hat u(\cdot))$ with respect to the supply rate function
\begin{equation*}
\omega(t,y,u)= f^0(t,y,u)-f^0(t,\hat y(t), \hat u(t)),\qquad \forall(t,y,u)\in \mathbb R\times E\times F,
\end{equation*}
if there exists a locally bounded and bounded from below storage function $S:\mathbb R\times E\rightarrow\mathbb R$, $\Pi$-periodic in time,  such that
\begin{equation*}
S(\tau_0,y(\tau_0))+\int_{\tau_0}^{\tau_1}\omega(t,y(t),u(t))\,dt\geq S(\tau_1,y(\tau_1))\;\;\;\text{for all}\;\; t_0\leq \tau_0<\tau_1\leq t_1,
\end{equation*}
for any admissible pair $(y(\cdot),u(\cdot))$.  If, in addition, there exists a $\mathcal{K}$-class function $\alpha(\cdot)$ such that
\begin{equation*}
S(\tau_0,y(\tau_0))+\int_{\tau_0}^{\tau_1}\omega(t,y(t),u(t))\,dt\geq S(\tau_1,y(\tau_1))+\int_{\tau_0}^{\tau_1}\alpha\big(\|(y(t)-\hat y(t), u(t)-\hat u(t))
\|_{X\times U}\big)\,dt,
\end{equation*}
we say it is strictly dissipative with respect to a $\Pi$-periodic trajectory $(\hat y(\cdot), \hat u(\cdot))$.
\end{definition}

The notion of (strict) dissipativity with respect to a periodic trajectory in Definition \ref{perdis} allows one to identify the optimal control problem $(\bar P_{[t_0,t_1]})$ as a periodic one. 
Consider the periodic optimal control problem 
$$
(\bar P_{\textrm{per}}) \qquad \left\{\begin{array}{l}
\bar J_{\textrm{per}}= \inf\frac{1}{\Pi}C_{[0,\Pi]}(y(\cdot),u(\cdot)),\\[2mm]
\text{subject to}\;\;\;
\dot y(t) = A(t)y+f(t,y(t),u(t)),\quad (y(t),u(t))\in E\times F,\;\;t\in[0,\Pi],\\[2mm]
y(0)=y(\Pi).\\
\end{array}\right.
$$
We assume that $(\bar P_{\textrm{per}})$ has at least one periodic optimal solution $(\bar y(\cdot),\bar u(\cdot))$ on $[0,\Pi]$ (see e.g. \cite{Barbu} for the existence of periodic  optimal solutions), and we set $\bar J_\textrm{per} = \frac{1}{\Pi}\int_0^\Pi f^0(t,\bar y(t),\bar u(t))\, dt$ (optimal value of $(\bar P_{\textrm{per}})$).   Let us extend $(\bar y(\cdot), \bar u(\cdot))$ in $\mathbb R$ by periodicity. Likewise, we have the following result.

\begin{proposition}%\label{periodic-example}
Assume that, for any $t_0$ and $t_1$, the problem $(\bar P_{[t_0,t_1]})$ is dissipative with respect to the $\Pi$-periodic optimal trajectory $(\bar y(\cdot), \bar u(\cdot))$, with the supply rate $\omega(t, y,u)=f^0(t,y,u)-f^0(t,\bar y(t), \bar u(t))$, and  the associated storage function $S(\cdot)$  is bounded on $E$ for all times. Then
\begin{itemize}
\item[(i).] $\bar J_{[t_0,+\infty)} = \bar J_{per}$, $\forall t_0\in \mathbb R$.
\item[(ii).] There exists a turnpike set $\mathcal{T} = \{ \bar y(t)\ \mid\ t\in[0,\Pi] \}$  such that 
$(H_1)$ and
 $(H_2)$ are satisfied.
\item[(iii).] Moreover, if  $(H_3)$ is satisfied and $(\bar P_{[t_0,t_1]})$ is strictly dissipative with respect to the \emph{$\Pi$-periodic} trajectory $(\bar y(\cdot), \bar u(\cdot))$,
with dissipation rate $\alpha(\cdot)$, then  $(H_4)$ is satisfied with $\beta(\cdot)=\alpha(\cdot)$ and $\text{dist}\, (y,\mathcal T)=\min_{t\in[0,\Pi]}\|y-\bar y(t)\|_X$.
 \end{itemize}
\end{proposition}

\begin{proof}
We only show the proof of $(i)$, as the rest is similar to the arguments in the proof of Proposition~\ref{thm3}.

Since $(\bar y(\cdot),\bar u(\cdot))$ is an admissible trajectory in $[t_0,t_0+k\Pi]$ for any $k\in\mathbb N$, we have
$\bar J_{[t_0,+\infty)} \leq \bar J_\textrm{per}$. Let us prove the converse inequality.
By the periodic dissipativity in Definition \ref{perdis}, we have
$$
S(t_0, y(t_0)) + \int_{t_0}^{t_0+k\Pi} f^0(t,y(t),u(t))\, dt \geq k \int_{t_0}^{t_0+\Pi} f^0(t,\bar y(t),\bar u(t))\, dt + S(t_0+k\Pi, y(t_0+k\Pi))$$
for any admissible trajectory $(y(\cdot), u(\cdot))$. Since $\bar J_\textrm{per} = \frac{1}{\Pi}\int_{0}^{\Pi} f^0(t,\bar y(t),\bar u(t))\, dt$, it follows that
$$
\bar J_\textrm{per} \leq \frac{1}{k\Pi}\int_{t_0}^{t_0+k\Pi} f^0(t,y(t),u(t))\, dt + \frac{S(t_0,y(t_0))-S(t_0+k\Pi,y(t_0+k\Pi))}{k\Pi} .
$$
Letting $k$ tend to infinity, and taking the infimum over all possible admissible trajectories, we get that $\bar J_\textrm{per} \leq \bar J_{[t_0,+\infty)}$.

\end{proof}

\section{Relationship with (strict) strong duality}\label{sec_dissip}

After having detailed a motivating example in Section \ref{guai2}, we recall in Section \ref{guai3}  the notion of (strict) strong 
duality, and we establish in Section \ref{guai4} that strict strong duality implies strict dissipativity (and thus measure-turnpike according to Section \ref{dis}).

\subsection{A motivating example}\label{guai2}
To illustrate the effect of Lagrangian  function associated with the static problem when one derives the measure-turnpike  property for the evolution control system, we consider the simplest model of heat equation with control constraints.

Let $\Omega\subset\mathbb R^n$,  $n\geq1$,  be a bounded domain with a smooth boundary $\partial \Omega$, and let $\mathcal D\subset\Omega$ be a non-empty 
 open subset. Throughout  this subsection, we denote by $\langle\cdot,\cdot\rangle$ and $\|\cdot\|$
the  inner product and norm in $L^2(\Omega)$, respectively; by $\chi_\mathcal D$ the characteristic function of $\mathcal D$. 
For any $T>0$, consider the optimal control  problem for the heat equation with pointwise control constraints:
$$\;\;\;\;\;\;\;\; \bar J^T=\inf_{u(\cdot)\in L^2(0,T;\,\mathcal U_{ad})}\frac{1}{2T}\int_0^T\Big(\|y(t)-y_d\|^2+\|u(t)\|^2\Big)\,dt$$
subject to
\begin{equation*}\left\{
\begin{split}
&y_t-\Delta y=\chi_{\mathcal D} u,\;\;\text{in}\;\;\Omega\times(0,T),\\
&y=0,\;\;\text{on}\;\;\partial\Omega\times(0,T),\\
&y(\cdot,0)=y_0,\;\;\text{in}\;\;\Omega,
\end{split}\right.
\end{equation*}
where $y_d\in L^2(\Omega)$, $y_0\in L^2(\Omega)$  and 
$$
\mathcal U_{ad}= \Big\{u\in L^2(\Omega)\;\,|\,\; u_a(x)\leq u(x)\leq u_b(x)\;\;\text{for a.e.}\;\;x\in \Omega\Big\}, 
$$
with $u_a$ and $u_b$ being in $L^2(\Omega)$. Assume that $(y^T(\cdot),u^T(\cdot))$ (the optimal pair obviously depends on the time horizon) is the unique optimal solution. We want to  study the long time behavior of optimal solutions,
i.e., the optimal pair stays in a neighborhood of a static optimal solution at most of the time horizon.

As before, we consider the static optimal control problem stated below
$$
\;\;\;\;\;\;\;\bar J_s=\inf_{u\in\mathcal U_{ad}}\frac{1}{2}\Big(\|y-y_d\|^2+\|u\|^2\Big)
$$
subject to
\begin{equation*}\left\{
\begin{split}
&-\Delta y=\chi_\mathcal D u,\;\;\text{in}\;\;\Omega,\\
&y=0,\;\;\text{on}\;\;\partial\Omega.\\
\end{split}\right.
\end{equation*}
Assume that $(y_s,u_s)$ is the unique optimal solution to $(P_s)$.

For this purpose, given every $\varepsilon>0$, we define the set
$$
Q_{\varepsilon, T}=\Big\{t\in[0,T]\;\;|\;\; \|y^T(t)-y_s\|^2+\|u^T(t)-u_s\|^2>\varepsilon \Big\},
$$ 
which measures the time at which the optimal pair is outside of the $\varepsilon$-neighborhood of $(y_s,u_s)$.

\begin{proposition}\label{he}
The following convergence hold
$$
\frac{1}{T}\int_0^Ty^T(t)\,dt\rightarrow y_s \;\;\text{and}\;\;\frac{1}{T}\int_0^Tu^T(t)\,dt\rightarrow u_s\;\;
\text{in}\;\;L^2(\Omega),\;\;\text{as}\;\;T\rightarrow\infty.
$$
Moreover, for each  $\varepsilon>0$, it holds true that 
$$
|Q_{\varepsilon,T}|\leq O\Big(\frac{1}{\varepsilon}\Big), \;\;\text{for all}\;\; T\geq1,
$$
i.e., the measure-turnpike property holds.
\end{proposition}

\begin{proof}
The key point of the proof is to show that
\begin{equation}\label{i7}
\int_0^T\Big(\|y^T(t)-y_s\|^2+\|u^T(t)-u_s\|^2\Big)\,dt\leq C \;\;\;\;\text{for all}\;\;T>0,
\end{equation}
where $C$ is a constant independent of $T$. Once the inequality \eqref{i7} is proved, the desired results hold 
 automatically. To prove \eqref{i7}, the remaining part of the proof is proceeded into several steps as follows.
 
{\bf Step 1}. We first introduce a Lagrangian function for the above stationary problem.
According to the Karusch-Kuhn-Tucker (KKT for short) optimality conditions (see, e.g.,  \cite[Theorem 2.29]{T}),  
there are functions $p_s$, $\mu_a$ and $ \mu_b$ in $ L^2(\Omega)$ such that 
\begin{equation*}(KKT)\;\;\;\left\{
\begin{split}
&-\Delta y_s=\chi_\mathcal D u_s,\;\;\;\; -\Delta p_s=y_d-y_s,      \;\;\text{in}\;\;\Omega,\\
&y_s=0,\;\;\;\;p_s=0,\;\;\text{on}\;\;\partial\Omega,\\
&u_s-\chi_\mathcal D p_s-\mu_a+\mu_b=0,\\
&\mu_a\geq0,\;\;\mu_b\geq0,\;\;\mu_a(u_a-u_s)=\mu_b(u_s-u_b)=0.\\
\end{split}\right.
\end{equation*}
Now,  we define the associated Lagrangian function $L:H_0^1(\Omega)\times L^2(\Omega)\rightarrow \mathbb R$ by setting
\begin{multline}\label{jia2}
L(y,u)=\frac{1}{2}\Big(\|y-y_d\|^2+\|u\|^2\Big)+\langle\nabla y,\nabla p_s\rangle-\langle\chi_\mathcal D u,p_s\rangle\\+
\langle\mu_a,u_a-u\rangle+\langle\mu_b,u-u_b\rangle,\;\;\forall (y,u)\in  H_0^1(\Omega)\times L^2(\Omega).
\end{multline}
From the above-mentioned KKT optimality conditions, we can see that
$$
L(y_s,u_s)=\frac{1}{2}\Big(\|y_s-y_d\|^2+\|u_s\|^2\Big)= \bar J_s,
$$
$$
L_{(y,u)}'(y_s,u_s)\big((y-y_s,u-u_s)\big)=0,
$$
$$
L^{''}_{(y,u)}(y_s,u_s)\big((y-y_s,u-u_s),(y-y_s,u-u_s)\big)=\|y-y_s\|^2+\|u-u_s\|^2.
$$
Since $L$ is a quadratic form, the Taylor expansion is 
  \begin{multline*}
 L(y,u)=L(y_s,u_s)+L_{(y,u)}'(y_s,u_s)\big((y-y_s,u-u_s)\big)\\
 +\frac{1}{2}L^{''}_{(y,u)}(y_s,u_s)\big((y-y_s,u-u_s),(y-y_s,u-u_s)\big),\;\;\forall (y,u)\in H^1_0(\Omega)\times L^2(\Omega),
  \end{multline*}
which means that
\begin{equation}\label{jia1}
L(y,u)=\bar J_s+\frac{1}{2}\Big(\|y-y_s\|^2+\|u-u_s\|^2\Big),\;\;\forall (y,u)\in H^1_0(\Omega)\times L^2(\Omega).
\end{equation}

{\bf Step 2}. 
Noting that $\mu_a\geq0$ and $\mu_b\geq0$, we obtain from \eqref{jia2} and \eqref{jia1} that for each $(y,u)\in H_0^1(\Omega)\times \mathcal U_{ad}$, 
\begin{multline}\label{i2}
\bar J_s+\frac{1}{2}\Big(\|y-y_s\|^2+\|u-u_s\|^2\Big)
\leq 
\frac{1}{2}\Big(\|y-y_d\|^2+\|u\|^2\Big)+\langle\nabla y,\nabla p_s\rangle-\langle\chi_\mathcal D u,p_s\rangle.
\end{multline}
Since  $(y^T(t),u^T(t))\in H^1_0(\Omega)\times \mathcal U_{ad}$ for a.e. $t\in(0,T)$, we get from \eqref{i2} that
\begin{multline*}
\bar J_s+\frac{1}{2}\Big(\|y^T(t)-y_s\|^2+\|u^T(t)-u_s\|^2\Big)
\leq 
\frac{1}{2}\Big(\|y^T(t)-y_d\|^2+\|u^T(t)\|^2\Big)\\+\langle\nabla y^T(t),\nabla p_s\rangle-\langle\chi_\mathcal D u^T(t),p_s\rangle,
\;\;\text{for a.e.}\;\;t\in(0,T).
\end{multline*}
Integrating the above inequality over $(0,T)$ and then multiplying the resulting by $1/T$, we have
\begin{multline}\label{huang1}
\bar J_s+\frac{1}{2T}\int_0^T\Big(\|y^T(t)-y_s\|^2+\|u^T(t)-u_s\|^2\Big)\,dt
\leq 
\bar J^T
+\frac{1}{T}\int_0^T\Big(\langle\nabla y^T(t),\nabla p_s\rangle-\langle\chi_\mathcal D u^T(t),p_s\rangle\Big)\,dt.
\end{multline}
Observe that 
\begin{equation*}\label{lu1}
-\big\langle y^T(T)-y_0,p_s\big\rangle=\int_0^T\Big(\big\langle\nabla y^T(t),\nabla p_s\big\rangle-\langle\chi_\mathcal D u^T(t),p_s\rangle\Big)\,dt.
\end{equation*}
This, along with \eqref{huang1}, implies that
\begin{equation}\label{i3}
\bar J_s+\frac{1}{2T}\int_0^T\Big(\|y^T(t)-y_s\|^2+\|u^T(t)-u_s\|^2\Big)\,dt\leq \bar J^T+
\frac{ \langle y_0-y^T(T),p_s\rangle }{T}.
\end{equation}
By the standard energy estimate for non-homogeneous heat equations, there is a constant $C>0$ (independent of $T>0$) such that
\begin{equation}\label{i5}
\|y^T(T)\|\leq C\Big(\|y_0\|+\max\big\{\|u_a\|,\|u_b\|\big\}\Big)\;\;\;\;\text{for all}\;\; T>0.
\end{equation}
Hence, by the Cauchy-Schwarz inequality we have
$$
\frac{ \langle y_0-y^T(T),p_s\rangle}{T}\leq \frac{C\|p_s\|}{T}\Big(\|y_0\|+\max\big\{\|u_a\|,\|u_b\|\big\}\Big)\leq O\Big(\frac{1}{T}\Big).
$$
This, together with \eqref{i3}, indicates that
\begin{equation}\label{huang2}
\bar J_s+\frac{1}{2T}\int_0^T\Big(\|y^T(t)-y_s\|^2+\|u^T(t)-u_s\|^2\Big)\,dt\leq \bar J^T+O\Big(\frac{1}{T}\Big).
\end{equation}

{\bf Step 3}. 
We claim that 
\begin{equation}\label{i10}
\bar J^T\leq \bar J_s+O\Big(\frac{1}{T}\Big), \;\;\text{when}\;\;T\geq1.
\end{equation}
Indeed, since $u_s$ is always an admissible control for the problem $(P^T)$, it holds that 
\begin{equation}\label{i6}
\bar J^T\leq \frac{1}{2T}\int_0^T\Big(\|y(t;u_s)-y_d\|^2+\|u_s\|^2\Big)\,dt,
\end{equation}
where $y(\cdot;u_s)$ is the solution to 
\begin{equation*}\left\{
\begin{split}
&y_t-\Delta y=\chi_\mathcal D u_s,\;\;\text{in}\;\;\Omega\times(0,T),\\
&y=0,\;\;\text{on}\;\;\partial\Omega\times(0,T),\\
&y(\cdot,0)=y_0,\;\;\text{in}\;\;\Omega.
\end{split}\right.
\end{equation*}
%Let $\delta y(t)= y(t;u_s)-y_s$, $t\in [0,T]$. We see that
%$$\delta y(t)=e^{t\Delta}(y_0-y_s),\;\;t\in[0,T].$$
It can be readily checked that 
\begin{equation}\label{iii4}
\frac{1}{2T}\int_0^T\Big(\|y(t;u_s)-y_d\|^2+\|u_s\|^2\Big)\,dt\leq \bar J_s+O\Big(\frac{1}{T}\Big), \;\;\text{when}\;\;T\geq1.
\end{equation}
Which in turn, together with \eqref{i6}, implies that  \eqref{i10}.

{\bf Step 4}. End of the proof for the inequality $\eqref{i7}$.
We obtain immediately from \eqref{huang2} and \eqref{i10} that
$$
\bar J^T=\bar J_s+O\Big(\frac{1}{T}\Big),
$$
as well as
$$
\frac{1}{2T}\int_0^T\Big(\|y^T(t)-y_s\|^2+\|u^T(t)-u_s\|^2\Big)\,dt\leq O\Big(\frac{1}{T}\Big),
$$ 
which is equivalent to the inequality \eqref{i7}.
\end{proof}

\begin{remark}\emph{
Notice that the inequality \eqref{i7} is stronger than the weak turnpike property \eqref{guai1} for $\mathcal T=\{y_s\}$.
The proof above yields the convergence result for the long-time horizon control problems towards
to the steady-state one in the measure-theoretical sense. It is an improved version of
the case of \textit{time-independent} controls \cite[Section 4]{PZ2}.} 
\end{remark}

\begin{remark}\emph{
We remark that, in the steps 2 and 3 of the proof of Proposition~\ref{he}, we have used the exponential stabilization of the heat equation to derive the upper bounds  \eqref{i5} and \eqref{iii4}. See also Remark~\ref{nonc}.}
\end{remark}

\subsection{What is (strict) strong duality }\label{guai3} 
In the above proof of Proposition  \ref{he},  we have seen an important role played by the Lagrangian  \eqref{jia1}, which is closely 
related to the notion of  strict strong duality introduced below. We recall that the notion of strong duality, well known in optimization (see, e.g.,  \cite{BV}).

\begin{definition}
We say that the static problem $(P_s)$ (in Section \ref{dd}) has the \emph{strong duality property} if there exists $\varphi_s\in D(A^*)$ (Lagrangian multiplier) such that $(y_s,u_s)$  minimizes the \emph{Lagrangian function} $L(\cdot,\cdot,\varphi_s): E\times  F\rightarrow \mathbb R$ defined by
$$L(y,u,\varphi_s)= f^0(y,u)+\langle A^*\varphi_s,y\rangle_{X^*,X}+\langle \varphi_s, f(y,u)\rangle_{X^*,X}.$$
We say  $(P_s)$ has the \emph{strict strong duality property} if there exists a $\mathcal K$-class function $\alpha(\cdot)$ such that
$$
L(y,u,\varphi_s)\geq L(y_s,u_s,\varphi_s)+\alpha\big(\|(y-y_s,u-u_s)\|_{X\times U}\big)
$$
for all $(y,u)\in E\times F$.
\end{definition}

\begin{remark}\emph{
Note that $L(y_s,u_s,\varphi_s)=\bar J_s$. 
If $(y_s,u_s)$ is the unique minimizer of the Lagrangian function $L(\cdot,\cdot,\varphi_s)$, and if $E\times F$ is compact in $X\times U$, then $(P_s)$ enjoys the strict strong duality property. However,  it is generally a very strong assumption that $L(\cdot,\cdot,\varphi_s)$ has a unique minimizer.  Note that  uniqueness of minimizers for elliptic
optimal control problems is still a long outstanding and difficult problem (cf., e.g., \cite{T}).}
\end{remark}

In finite dimension, strong duality is introduced and investigated in optimization problems for which the primal and dual problems are equivalent. The notion of strong duality is closely related to the saddle point property of the Lagrangian function associated with the primal optimization problem (see, e.g., \cite{BV,T}).
Note that Slater's constraint qualification (also known as ``interior point'' condition) is a sufficient condition ensuring strong duality for a convex problem, and note that, when the primal problem is convex, the well known Karusch-Kuhn-Tucker conditions are also sufficient conditions ensuring strong duality (see \cite[Chapter 5]{BV}).  
Similar assumptions are also considered for other purposes in the literature (see, for example, \cite[Assumption 1]{CHJ},   \cite[Assumption 4.2 (ii)]{CarlsonBOOK} and  \cite[Assumption 2]{strongduality}.) 

In infinite dimension, however, the usual strong duality theory (for example, the above-mentioned Slater condition) cannot be applied  because the underlying constraint set may have an empty interior. The corresponding strong dual theory, as well as the existence of Lagrange multipliers associated to optimization problems or to variational inequalities, have been developed only quite recently in \cite{Dstrong}.
The strict strong duality property is closely related to the second-order sufficient optimality condition, which guarantees the local 
optimality of $(y_s,u_s)$ for the problem $(P_s)$ (see, e.g., \cite{T}).

%\subsection{Examples}
We provide hereafter two  examples satisfying the strict strong duality property.

\begin{example}\label{cont}
Consider the static optimal control problem
\begin{equation*}
\qquad \left\{\begin{split}
& \inf \,J_s(y,u)= f^0(y,u), \\
& \text{subject to}\;\;\;\;\;Ay+Bu=0, \\
& y\in E, \quad  u\in F,\\
\end{split}\right.
\end{equation*}
with $A\in\mathbb R^{n\times n}, B\in\mathbb R^{n\times m}$,
$f^0(\cdot,\cdot)$ a strictly convex function, $E$ and $F$ convex, bounded and closed subsets of $\mathbb R^n$ and of $\mathbb R^m$, respectively. 

Assume that \emph{Slater's condition} holds, i.e., there exists an interior point $(\tilde{y},\tilde {u})$ of $E\times F$ such that $A\tilde y+B\tilde u=0$.  Recall that the  Lagrangian function $L:E\times F\times\mathbb R^n\rightarrow\mathbb R$  is given  by
$L(y,u,\varphi)= f^0(y,u)+\langle \varphi,Ay+Bu\rangle_{\mathbb R^n}$.
Let $(y_s,u_s)$ be the unique optimal solution.
It follows from the Slater condition that there exists a Lagrangian multiplier  $\varphi_s\in\mathbb R^{n}$ such that (see, e.g., \cite[Section 5.2.3]{BV})
\begin{equation*}\label{5232}
L(y,u,\varphi_s)>L(y_s,u_s,\varphi_s),\;\;\forall (y,u)\in E\times F\setminus\{(y_s,u_s)\}.
\end{equation*}
The strict inequality is due to the strict convexity of the cost function $f^0$. Setting
$$\widetilde{L}(y,u)= L(y,u,\varphi_s)-L(y_s,u_s,\varphi_s),\;\;\forall (y,u)\in E\times F,$$
we have $\widetilde{L}(y_s,u_s)=0$ and $\widetilde{L}(y,u)>0$ for all $(y,u)\in E\times F \setminus\{(y_s,u_s)\}$.

We claim that
\begin{equation}\label{in5211}
\widetilde{L}(y,u)\geq \alpha\big(\|(y-y_s,u-u_s)\|_{\mathbb R^{n+m}}\big),\;\;\forall (y,u)\in E\times F
\setminus\{(y_s,u_s)\},
\end{equation}
for some  $\mathcal K$-class function $\alpha(\cdot)$.
Indeed, since $E\times F$ is  compact  in $ \mathbb R^{n+m}$, without loss of generality, 
we assume that $E \times F\subset B_r(y_s,u_s)$ with $r>0$, where 
$$B_r(y_s,u_s)=\big\{(y,u)\in \mathbb R^{n+m}\ \mid\ \|(y-y_s,u-u_s)\|_{\mathbb R^{n+m}}\leq r\big\}.$$
Since the function $\widetilde{L}(\cdot,\cdot)$ is continuous, we define
\begin{equation*}
\alpha(\gamma)= \inf_{\substack{(y,u)\in E\times F\\
\gamma\leq\|(y-y_s,u-u_s)\|_{\mathbb R^{n+m}}\leq r}}\widetilde{L}(y,u),\;\;\;\;\;\text{when}\;\gamma\in[0,r], \end{equation*}
and $\alpha(\gamma)\equiv\alpha(r)$ when $\gamma>r$.
It is easy to check that the inequality \eqref{in5211} holds with the $\mathcal K$-class function $\alpha(\cdot)$ given above. This means that the static problem has the strict strong duality property.
\end{example}

\medskip

\begin{example}\label{ex1}
Let $\Omega\subset\mathbb R^3$ be a bounded domain with a smooth boundary $\partial \Omega$.
Given any $y_d\in L^2(\Omega)$, we consider the static optimal control problem
$$
\;\;\;\;\inf \frac{1}{2}\big(\|y-y_d\|^2_{L^2(\Omega)}+\|u\|^2_{L^2(\Omega)}\big),
$$
over all $(y,u)\in H^{1}_{0}(\Omega)\times L^2(\Omega)$ satisfying
\begin{equation*}
\left\{
\begin{split}
&-\triangle y+y^3=u\;\;&\text{in}\;\;\Omega,\\
&y=0\;\;&\text{on}\;\;\partial\Omega.
\end{split}
\right.
\end{equation*}
Let $(y_s,u_s)$ be an optimal solution of this problem. According to first-order necessary optimality conditions (see, e.g., \cite[Chapter 1]{Kun1} or \cite[Chapter 6, Section 6.1.3]{T}), there exists an adjoint state $\varphi_s\in H^2(\Omega)\cap H^1_0(\Omega)$ satisfying
\begin{equation*}
\left\{
\begin{split}
&-\triangle \varphi_s+3y_s^2\varphi_s=y_s-y_d\;\;&\text{in}\;\;\Omega,\\
&\varphi_s=0\;\;&\text{on}\;\;\partial\Omega ,
\end{split}
\right.
\end{equation*}
such that $u_s=\varphi_s$. Moreover, since $\varphi_s$ is a Lagrangian multiplier associated with $(y_s,u_s)$ for the Lagrangian function $L(\cdot,\cdot,\varphi_s): H^1_0(\Omega)\times L^2(\Omega) \rightarrow\mathbb R$ defined by 
\begin{equation*}
L(y,u,\varphi_s)=\frac{1}{2}\big(\|y-y_d\|^2_{L^2(\Omega)}+\|u\|^2_{L^2(\Omega)}\big)+\langle -\triangle\varphi_s, y\rangle_{L^2(\Omega),L^2(\Omega)}+
\langle \varphi_s, y^3-u\rangle_{L^2(\Omega),L^2(\Omega)},
\end{equation*}
we have
\[
L(y_s,u_s,\varphi_s)\leq L(y,u,\varphi_s), \;\;\forall (y,u)\in H_0^1(\Omega)\times L^2(\Omega). 
\]
It means that $(P_s)$ has the strong duality property.

Next, we claim that it holds the strict strong duality property under the condition that 
$\|y_d\|_{L^2(\Omega)}$ is small enough. 
Notice that 
$$\frac{1}{2}\big(\|y_s-y_d\|^2_{L^2(\Omega)}+\|u_s\|^2_{L^2(\Omega)}\big)\leq \frac{1}{2}\|y_d\|^2_{L^2(\Omega)}.$$
Now, assuming that the norm of the target $y_d$ is small enough guarantees the smallness of $(y_s,u_s)$, which consequently belongs to a ball $B_r$ in $H_0^1(\Omega)\times L^2(\Omega)$, centered at the origin and with a small radius $r>0$. Moreover, by elliptic regularity, we deduce that the norms of $y_s$ and $\varphi_s$ are small in $ H^2(\Omega)\cap L^\infty(\Omega)$ (see \cite[Section 3]{PZ2}). 
For the Lagrangian function $L(\cdot,\cdot,\varphi_s)$ defined above,
its first-order Fr\'echet derivative  is
\begin{equation}\label{c4291}
L'(y_s,u_s,\varphi_s)\left((y-y_s,u-u_s)\right)=0,
\end{equation}
and its second-order Fr\'echet derivative is
\begin{multline}\label{c4292}
L''(y_s,u_s,\varphi_s)\left((y-y_s,u-u_s),(y-y_s,u-u_s)\right)\\
=\|y-y_s\|^2_{L^2(\Omega)}+\|u-u_s\|^2_{L^2(\Omega)}+6\int_\Omega y_s\varphi_s(y-y_s)^2\,dx,
\end{multline}
whenever $(y,u)\in B_r$ (see, for instance, \cite[Chapter 6, pp. 337-338]{T}).
Note that
\begin{multline*}
\begin{split}
L(y,u,\varphi_s)&=L(y_s,u_s,\varphi_s)+L'(y_s,u_s,\varphi_s)\left((y-y_s,u-u_s)\right) \\
&\;\;\;+L''(y_s,u_s,\varphi_s)\left((y-y_s,u-u_s),(y-y_s,u-u_s)\right)\\
&\;\;\;+o(\|y-y_s\|_{L^2(\Omega)}^2+\|u-u_s\|^2_{L^2(\Omega)}),
\end{split}
\end{multline*}
for all $(y,u)\in B_r$. This, together with
\eqref{c4291}, \eqref{c4292} and the smallness of $(y_s,\varphi_s)$ in $L^\infty(\Omega)$,
implies that
\begin{equation*}
L(y,u,\varphi_s)\geq L(y_s,u_s,\varphi_s)+\frac{1}{2}(\|y-y_s\|_{L^2(\Omega)}^2+\|u-u_s\|^2_{L^2(\Omega)}), \;\;\;\forall (y,u)\in B_r,
\end{equation*}
which proves the above claim.
\end{example}

\begin{remark}
\emph{
Similar to second order gap conditions for local optimality \cite{T},  the positive semi-definiteness of Hessian matrix of the Hamiltonian  is a necessary condition for the local optimality, 
while its positive definiteness is a sufficient condition for the local optimality. The latter is also known as  the strengthened
Legendre-Clebsch condition.}

\end{remark}

%%%%%%%%%%%%%%%%%%%%%%%%%%%%%%%

\subsection{Strict strong duality implies strict dissipativity}\label{guai4} 

In this subsection,  
by means of strict strong duality, we extend  Proposition \ref{he} to  general optimal control problems. 
More precisely,
we  establish sufficient conditions, in terms of (strict) strong duality for $(P_s)$, under which (strict)  dissipativity  holds true with a specific storage function for $(\bar P_{[0,T]})$  in Section \ref{dd}.  As seen in Theorem~\ref{turnpikeproperty1}, strict dissipativity implies measure-turnpike. 

\begin{theorem}\label{equiv}
Let $E$ be a bounded subset of $X$. Then,
strong duality (resp., strict strong duality) for  $(P_s)$ implies dissipativity (resp., strict dissipativity) for $(\bar P_{[0,T]})$, with the storage function given by $S(y)=-\langle \varphi_s,y\rangle_{X^*,X}$ for every $y\in E$.  Consequently,  $(\bar P_{[0,T]})$ has the measure-turnpike property under the strict strong duality property.
\end{theorem}

\begin{proof}%[Proof of Theorem \ref{equiv}]
It suffices to prove that strong duality for $(P_s)$ implies dissipativity for $(\bar P_{[0,T]})$ (the proof with the ``strict" additional property is similar with only minor modifications).

By the definition of strong duality, there exists a Lagrangian multiplier $\varphi_s\in D(A^*)$
such that $L(y_s,u_s,\varphi_s)\leq L(y,u,\varphi_s)$ for all $(y,u)\in E\times F$, which means that
\begin{equation*}
f^0(y_s,u_s)\leq f^0(y,u)+\langle A^*\varphi_s,y\rangle_{X^*,X}+\langle \varphi_s,f(y,u)\rangle_{X^*,X}\qquad\forall(y,u)\in E\times F.
\end{equation*}
Let $T>0$. Assume that $(y(\cdot),u(\cdot))$ is  an admissible pair for the problem  $(\bar P_{[0,T]})$. Then,
\begin{equation*}
f^0(y_s,u_s)\leq f^0(y(t),u(t))+
\langle A^*\varphi_s,y(t)\rangle_{X^*,X}+\langle \varphi_s,f(y(t),u(t))\rangle_{X^*,X},\;\;\text{for a.e.}\ t\in[0,T].
\end{equation*}
Integrating the above inequality over $(0,\tau)$, with $0<\tau\leq T$, leads to 
\begin{equation}\label{ma1}
\tau f^0(y_s,u_s)\leq \int_0^\tau f^0(y(t),u(t))\,dt+\int_0^\tau\langle A^*\varphi_s,y(t)\rangle_{X^*,X}\,dt
+ \int_0^\tau \langle \varphi_s,f(y(t),u(t))\rangle_{X^*,X}\,dt.
\end{equation}
Notice that $(y(\cdot), u(\cdot))$ satisfies the state  equation in the problem $(\bar P_{[0,T]})$, we have
\begin{equation*}
\int_0^\tau\langle A^*\varphi_s,y(t)\rangle_{X^*,X}\,dt
+ \int_0^\tau \langle \varphi_s,f(y(t),u(t))\rangle_{X^*,X}\,dt
=\langle \varphi_s,y(\tau)\rangle_{X^*,X}-\langle \varphi_s,y(0)\rangle_{X^*,X}.
\end{equation*}
This, together with \eqref{ma1}, leads to 
\begin{equation*}
 \int_0^\tau \Big(f^0(y(t),u(t))-f^0(y_s,u_s)\Big)\,dt+\langle \varphi_s, y(\tau)\rangle_{X^*,X}
\geq \langle \varphi_s,y(0)\rangle_{X^*,X}.
\end{equation*}
Set $S(y)=-\langle \varphi_s,y\rangle_{X^*,X}$ for every $y\in E$. Since $E$ is a bounded subset of $X$,  we see that $S(\cdot)$ is locally bounded and bounded from below. 
Therefore, we infer that $\{(\bar P_{[0,T]})\,\mid\,T>0\}$ has the dissipativity property.
\end{proof}

\begin{remark}
\emph{Strong duality and dissipativity are equivalent in some situations:}
\begin{itemize}
\item \emph{On one hand, we proved above that  strong duality (resp. strict strong duality) implies dissipativity (resp., strict dissipativity). We refer also the reader to \cite[Lemma 3]{Faulwasser1} for a closely related result. }

\item  \emph{On the other hand, it is easy to see that, if the storage function $S(\cdot)$ is continuously Fr\'echet differentiable, then strong duality (resp., strict strong duality) is the infinitesimal version of the dissipative inequality \eqref{5224} (resp., of \eqref{5225}).
For this point, we also mention that \cite[Assumption 5.2]{Grune2}  is a discrete version of strict dissipativity, and  that \cite[Inequality (14)]{Faulwasser1} is the infinitesimal version of strict dissipativity for the continuous system when the storage function is differentiable.}
\end{itemize}
\end{remark}

%%%%%%%%%%%%%%%%%%%%%%%

%%%%%%%%%%%%%%%%%%%%%%%%%%%%%%%%%%%%%%%%%%%%%%%%%%%%%%%%%%%

\section{Conclusions and further comments}\label{consec}
In this paper,  we first have proved a general turnpike phenomenon around a set holds 
for optimal control problems with terminal state constraints in an abstract framework.
Next, we have obtained the following auxiliary result:
\begin{quote}
strict strong duality $\Rightarrow$ strict dissipativity $\Rightarrow$ measure-turnpike property.
\end{quote}
We have also used dissipativity to identify the long-time limit of optimal values.

\medskip

Now, several comments and perspectives are in order.
\paragraph{Measure-turnpike versus exponential turnpike.}
In the paper \cite{TZZ}, we establish the exponential turnpike property for general classes of optimal control problems in infinite dimension that are similar to the problem $(\bar P_{[0,T]})$ investigated in the present paper, but with the following differences:
\begin{itemize}
\item[(i)] $E=X$ and $F=U$;
%\item[(ii)] $X$ is a Banach space with a dual $X'$ that is strictly convex, and $U$ is a separable metric space;
\item[(ii)] $y(0)=y_0\in X$.
\end{itemize}
The item (i) means that, in \cite{TZZ}, we consider optimal control problems without any state or control constraint. Under the additional assumption made in (ii), we are then able to apply the \emph{Pontryagin maximum principle} in Banach spaces (see \cite{LiXunjing}), thus obtaining an extremal system that is \emph{smooth}, which means in particular that the extremal control is a \emph{smooth} function of the state and of the adjoint state. This smooth regularity is crucial in the analysis done in \cite{TZZ} (see also \cite{TZ1}), consisting of linearizing the extremal system around an equilibrium point, which is itself the optimal solution of an associated static optimal control problem, and then of analyzing from a spectral point of view the hyperbolicity properties of the resulting linear system. Adequately interpreted, this implies the local exponential turnpike property, saying that
$$
\left\Vert y^T(t)-y_s\right\Vert_X+\left\Vert u^T(t)-u_s\right\Vert_U+\left\Vert \lambda^T(t)-\lambda_s\right\Vert_X\leq c \left( e^{-\mu t}+e^{-\mu(T-t)} \right) ,
$$
for every $t\in[0,T]$, for some constants $\mu,c>0$ not depending on $T$, where $\lambda^T$ is the adjoint state coming from the Pontryagin maximum principle.
There are many examples of control systems for which the measure-turnpike holds but not exponential turnpike (see, for instance, Example ~\ref{cont}).
The exponential turnpike property is much stronger than the measure-turnpike property, not only because it gives an exponential estimate on the control and the state, instead of the softer estimate in terms of Lebesgue measure, but also because it gives the closeness property for the adjoint state. This leads us to the next comment.

\red{}

\paragraph{Turnpike on the adjoint state.}
As mentioned above, the exponential turnpike property established in \cite{TZZ} holds as well for the adjoint state coming from the application of the Pontryagin maximum principle. This property is particularly important when one wants to implement a numerical shooting method in order to compute the optimal trajectories. Indeed, the exponential closeness property of the adjoint state to the optimal static adjoint allows one to successfully initialize a shooting method, as chiefly explained in \cite{TZ1} where an appropriate modification and adaptation of the usual shooting method has been described and implemented.

The flaw of the linearization approach developed in \cite{TZZ} is that it does not a priori allow to take easily into account some possible control constraints (without speaking of state constraints).

The softer approach developed in the present paper leads to the weaker property of measure-turnpike, but permits to take into account some state and control constraints.

However, under the assumption (ii) above, one can as well apply the Pontryagin maximum principle, and thus obtain an adjoint state $\lambda^T$. Due to state and control constraints, of course, one cannot expect that the extremal control $u^T$ be a smooth function of $y^T$ and $\lambda^T$, but anyway our approach by dissipativity is soft enough to yield the measure-turnpike property for the optimal state $y^T$ and for the optimal control $u^T$. Now, it is an open question to know whether the measure-turnpike property holds or not for the adjoint state $\lambda^T$.
As mentioned above, having such a result is particularly important in view of numerical issues.

%\red{Are there situations where turnpike for the state is known without having it also for the adjoint?}

\paragraph{Local versus global properties.}
It is interesting to stress on the fact that Theorem \ref{turnpikeproperty1} (saying that strict dissipativity implies measure-turnpike) is of \emph{global} nature, whereas Theorem \ref{equiv} (saying that strict strong duality implies strict dissipativity) is rather of \emph{local} nature. This is because, as soon as Lagrangian multipliers enter the game, except under strong convexity assumptions this underlies that one is performing reasonings that are local, such as applying first-order conditions for optimality. Therefore, although Theorem \ref{equiv} provides a sufficient condition ensuring strict dissipativity and thus allowing one to apply the result of Theorem \ref{turnpikeproperty1}, in practice showing strict strong duality can in general only be done locally. In contrast, dissipativity is a much more general property, which is global in the sense that it reflects a global qualitative behavior of the dynamics, as in the Lyapunov theory. We insist on this global picture because this is also a big difference with the results of \cite{TZZ,TZ1} on exponential turnpike, that are purely local and require smallness conditions. Here, in the framework of Theorem \ref{turnpikeproperty1}, no smallness condition is required. The price to pay however is that one has to know a storage function, ensuring strict dissipativity. In practical situations this is often the case and storage functions often represent an energy that has a physical meaning.

\paragraph{Semilinear heat equation.}
We end the paper with a still open problem, related to the above-mentioned smallness condition. Continuing with  Example \ref{ex1}, given any $y_d\in L^2(\Omega)$ we consider the evolution optimal control problem
\begin{equation*}\label{semilinheat1}
\;\;\;\inf\, \frac{1}{2T}\int_0^T \left( \|y(t)-y_d\|^2_{L^2(\Omega)}+\|u(t)\|_{L^2(\Omega)}^2\right) dt
\end{equation*}
over all possible solutions of
\begin{equation}\label{semilinheat2}
\left\{
\begin{split}
&y_t-\triangle y+y^3=u\;\;&\text{in}\;\;\Omega\times(0,T),\\
&y=0\;\;&\text{on}\;\;\partial\Omega\times(0,T),
\end{split}
\right.
\end{equation}
such that $(y(t),u(t))\in H_0^1(\Omega)\times L^2(\Omega)$ for almost every $t\in (0,T)$. 
It follows from Example~\ref{ex1} and Theorem~\ref{equiv}  that the  problem  is dissipative
at an optimal  stationary point $(y_s,u_s)$ with the storage function $S(y)=-\langle \varphi_s,y\rangle_{L^2(\Omega),L^2(\Omega)}$. Under the
additional  smallness condition on $\|y_d\|_{L^2(\Omega)}$, the strict strong duality holds and thus 
the measure-turnpike property follows.   
As said above, this assumption reflects the fact that Theorem~\ref{equiv} is rather of a local nature. However, due to the fact that the nonlinear term in \eqref{semilinheat2} has the ``right sign", we do not know how to take advantage of this monotonicity of the control system \eqref{semilinheat2} to infer the measure-turnpike property.
It is interesting to compare this result with \cite[Theorem 3.1]{PZ2}, where the authors used a subtle analysis of optimality systems  to establish an exponential turnpike property, under the same smallness condition. 
The question of whether the turnpike property actually holds or not for optimal solutions \emph{but} without the smallness condition on the target,  is still an interesting open problem.

\bigskip

\noindent \textbf{Acknowledgment}.
We would like to thank Prof. Enrique Zuazua for fruitful discussions and valuable suggestions on this subject. 
We acknowledge the financial support by the grant FA9550-14-1-0214 of the EOARD-AFOSR. The second author was partially supported by the National Natural Science Foundation of China under grants 11501424 and 11371285.

\end{document}